\documentclass[letterpaper, 10 pt]{IEEEtran}
\IEEEoverridecommandlockouts
\usepackage{cite}
\usepackage{amsmath,amssymb,amsfonts,amsthm}
\usepackage{mathtools}
\usepackage{abraces}
\usepackage{algorithmic}
\usepackage{subfig}
\usepackage{graphicx}
\usepackage{svg}
\usepackage{textcomp}
\usepackage{pgfplots}
\usepgfplotslibrary{groupplots}

\usepackage{bm}
\usepackage{hyperref}
\let\labelindent\relax
\usepackage{enumitem}
\usepackage{siunitx}
\usepackage{physics}
\AtBeginDocument{\RenewCommandCopy\qty\SI}

\newcounter{thms}
\newtheorem{theorem}[thms]{Theorem}
\newtheorem{corollary}[thms]{Corollary}
\newtheorem{lemma}[thms]{Lemma}
\newtheorem{proposition}[thms]{Proposition}
\newcounter{asses}
\newtheorem{assumption}[asses]{Assumption}
\newcounter{rems}
\newtheorem{remark}[rems]{Remark}

\DeclareMathAlphabet{\mymathbb}{U}{BOONDOX-ds}{m}{n}

\newcommand{\bb}[1]{\mathbb{#1}}

\newcommand{\mc}[1]{\mathcal{#1}}
\newcommand{\mt}[1]{\textnormal{#1}}

\newcommand{\lt}{\left}
\newcommand{\rt}{\right}
\newcommand{\beeq}{\begin{equation}}
\newcommand{\eneq}{\end{equation}}
\newcommand{\matb}{\begin{matrix}}
\newcommand{\mate}{\end{matrix}}

\DeclareMathOperator*{\argmax}{arg\,max}
\DeclareMathOperator*{\argmin}{arg\,min}

\DeclareMathOperator*{\st}{subject~to}

\makeatletter
\def\smalloverbrace#1{\mathop{\vbox{\m@th\ialign{##\crcr\noalign{\kern3\p@}%
  \tiny\downbracefill\crcr\noalign{\kern3\p@\nointerlineskip}%
  $\hfil\displaystyle{#1}\hfil$\crcr}}}\limits}
\makeatother

\definecolor{redtmp}{rgb}{0, 0, 0}
\definecolor{redtmp2}{rgb}{0.93, 0.19, 0.10}
\definecolor{myr}{rgb}{0.93, 0.19, 0.10}

\begin{document}

\title{\LARGE Distributionally robust infinite-horizon control: from a pool of samples to the design of dependable controllers}

\author{Jean-Sébastien Brouillon\textsuperscript{$\star$}, Andrea Martin\textsuperscript{$\star$}, John Lygeros, Florian D\"orfler, and Giancarlo Ferrari-Trecate
\thanks{Jean-Sébastien Brouillon, Andrea Martin, and Giancarlo Ferrari-Trecate are with the Institute of Mechanical Engineering, EPFL, Switzerland. E-mail addresses: \{jean-sebastien.brouillon, andrea.martin, giancarlo.ferraritrecate\}@epfl.ch.}
\thanks{John Lygeros and Florian D\"orfler are with the Department of Information Technology and Electrical Engineering, ETH Z\"urich, Switzerland. E-mail addresses: \{jlygeros, dorfler\}@ethz.ch.}
\thanks{Research supported by the Swiss National Science Foundation (SNSF) under the NCCR Automation (grant agreement 51NF40\textunderscore 180545).}
\thanks{\textsuperscript{$\star$}Jean-Sébastien Brouillon and Andrea Martin contributed equally.}
}

\maketitle

\begin{abstract}
We study control of constrained linear systems with only partial statistical information about the uncertainty affecting the system dynamics and the sensor measurements. Specifically, given a finite collection of disturbance realizations drawn from a generic distribution, we consider the problem of designing a stabilizing control policy with provable safety and performance guarantees despite the mismatch between the empirical and true distributions. We capture this discrepancy using Wasserstein ambiguity sets, and we formulate a distributionally robust (DR) optimal control problem, which provides guarantees on the expected cost, safety, and stability of the system. To solve this problem, we first present new results for DR optimization of quadratic objectives using convex programming, showing that strong duality holds under mild conditions. Then, by combining our results with the system-level parametrization of linear feedback policies, we show that the design problem can be reduced to a semidefinite program. We present numerical simulations to validate the effectiveness of our approach and to highlight the value of empirical distributions for control design.
\end{abstract}


\section{Introduction}
As modern engineered systems become increasingly complex and interconnected, classical control methods based on stochastic optimization face the challenge of overcoming the lack of a precise statistical description of the uncertainty. In fact, the probability distribution of the uncertainty is often unknown and only indirectly observable through a finite number of independent samples. In addition, replacing the true distribution with a nominal estimate in the spirit of certainty equivalence often proves unsatisfactory; the optimization process amplifies any statistical error in the distribution inferred from data, resulting in solutions that are prone to poor out-of-sample performance \cite{mohajerin2018data, kuhn2019wasserstein, shafieezadeh2023new}.

The paradigm of distributionally robust optimization (DRO) considers a minimax stochastic optimization problem over a neighborhood of the nominal distribution defined in terms of a distance in the probability space. In this way, the solution becomes robust to the most averse distribution that is sufficiently close to the nominal one, while the degree of conservatism of the underlying optimization can be regulated by adjusting the radius of the ambiguity set. 

While several alternatives have been proposed to measure the discrepancy between probability distributions, including the Kullback–Leibler divergence and the total variation distance \cite{gibbs2002choosing}, recent literature has shown that working with ambiguity sets defined using the Wasserstein metric \cite{villani2009optimal} offers a number of advantages in terms of expressivity, computational tractability, and statistical out-of-sample guarantees \cite{mohajerin2018data, kuhn2019wasserstein, shafieezadeh2023new}. Thanks to these properties, Wasserstein DRO has found application in a wide variety of domains, ranging from finance and machine learning to game theory, see, e.g., \cite{kuhn2019wasserstein, shafieezadeh2023new, chen2022data, frogner2015learning, adu2022optimal}. 

Wasserstein ambiguity sets have recently been interfaced with the dynamic environments and continuous actions spaces typical of control. In \cite{taskesen2024distributionally}, the authors consider a generalization of classical linear quadratic Gaussian (LQG) control, where the noise distributions belong to Wasserstein balls centered at nominal Gaussian distributions. Motivated by the idea of leveraging uncertainty samples for data-driven decision-making under general distributions, a parallel line of research instead considers ambiguity sets centered at nominal empirical distributions. Among other contributions exploiting the greater expressivity provided by this data-driven approach, \cite{mark2020stochastic, fochesato2022data, aolaritei2023wasserstein, pilipovsky2024distributionally, coulson2021distributionally} consider the design of tube-based predictive control schemes, \cite{yang2020wasserstein, kargin2024wasserstein, kim2023distributional, taskesen2024distributionally} address infinite-horizon problems using dynamic programming and duality theory, \cite{krishnan2020probabilistic} and \cite{brouillon2023regularization} focus on filtering and state estimation problems. More fundamentally, \cite{aolaritei2022uncertainty} and \cite{aolaritei2023capture} provide exact characterizations of how Wasserstein ambiguity sets propagate through the system dynamics, shedding light on the role of feedback in controlling shape and size of the ambiguity sets resulting from distributional uncertainty. Despite these advances, it remains unclear how the availability of samples can drive the design of a control policy that guarantees safety, performance, and stability in the face of distributional uncertainty. 

Motivated by this challenge, this paper makes the following contributions. We first establish novel strong duality results for DRO of quadratic functions that, differently from the results in \cite{kuhn2019wasserstein, shafieezadeh2023new}, hold even if the support of the uncertainty is bounded. Then, leveraging the system-level parametrization (SLP) of linear dynamic controllers \cite{wang2019system}, we present a convex reformulation of the Distributionally Robust Infinite-horizon Controller (DRInC) synthesis problem, which exploits a finite impulse response (FIR) approximation of the system closed-loop maps. 
As key advantages, our optimization-based approach guarantees stability of the closed-loop interconnection by design, and only requires one-shot offline computations, bypassing the computational bottleneck of recomputing the optimal control policy online with a receding horizon strategy \cite{mark2020stochastic, fochesato2022data, aolaritei2023wasserstein}. In fact, as the complexity of the synthesis problem increases with the number of considered uncertainty samples, solving the policy optimization problem in real time becomes prohibitive whenever the nominal empirical distribution is estimated using a sufficiently large number of uncertainty samples.
Further, differently from \cite{yang2020wasserstein} and \cite{kim2023distributional}, which consider infinite-horizon DR control in unconstrained scenarios, our approach naturally extends to include satisfaction of probabilistic safety specifications expressed as DR conditional value-at-risk (CVaR) constraints. Lastly, the proposed optimization perspective allows us to seamlessly study the partially observed setting, extending the recent results \cite{coulson2021distributionally, hakobyan2022wasserstein, taskesen2024distributionally} on output-feedback DR control to the infinite horizon case. As we comment throughout the paper, our formulation encompasses several control problems considered in the literature, providing a unified perspective on stochastic and robust control objectives.

\section{Problem Statement}
\subsection{System dynamics and uncertainty description}\label{subsec_system_dyn_and_uncertainty}
We consider controllable and observable linear dynamical systems described by the state-space equations:
\begin{equation}
\label{eq:system_dynamics_definition}
    x_{t+1} = A x_t + B u_t + w_t\,, ~ y_t = C x_t + v_t\,,
\end{equation}
where $x_t \in \mathbb{R}^n$, $u_t \in \mathbb{R}^m$, $y_t \in \mathbb{R}^p$, $w_t \in \mathbb{R}^n$, and $v_t \in \mathbb{R}^p$ are the system state, the control input, the measured output, and the stochastic disturbances modeling process and measurement noise, respectively. 
We study infinite-horizon control when only partial statistical information about the distribution of the joint disturbance process $\xi_t = [w_t^\top, v_t^\top]^\top$ is available. Specifically, we assume availability of $N \in \mathbb{N}$ independent observations $\bm{\xi}_{T}^{(1)}, \dots, \bm{\xi}_{T}^{(N)}$, where each sample
\begin{align}
    \label{eq:training_samples}
    \bm{\xi}_{T}^{(i)} &= \lt[{\bm{w}_{T}^{(i)}}^{\!\top} \!\!, {\bm{v}_{T}^{(i)}}^{\!\top} \rt]^{\!\top} \!\!\!\!,\nonumber\\
    &= \lt[{w^{(i)}_0}^{\!\top} \!\!, \dots, {w^{(i)}_{T}}^{\!\top}, {v^{(i)}_0}^{\!\top} \!\!, \dots, {v^{(i)}_{T}}^{\!\top} \rt]^{\!\top} \!\!\!,
\end{align}
comprises a trajectory of length $T \in \mathbb{N}$ of $w_t$ and $v_t$. As no performance or safety guarantee can be established if the samples in \eqref{eq:training_samples} are not representative of the asymptotic statistics of $\bm{w}$ and $\bm{v}$, we introduce the following stationarity assumption \cite[p. 154]{park2018fundamentals}.
\begin{assumption}
\label{ass:noise_follows_stationary_process}
    For all $t \in \mathbb{N}$, the stochastic process that generates the joint disturbance vector $\xi_t = [w_t^\top, v_t^\top]^\top$ is \emph{stationary of order $T$}, i.e., $\mathbb{P}(\xi_0, \dots, \xi_T) = \mathbb{P}(\xi_t, \dots, \xi_{t+T})$.
\end{assumption}
We note that Assumption~\ref{ass:noise_follows_stationary_process} subsumes the usual setting where the disturbance processes are independent and identically distributed, and allows modeling temporal correlation between samples that are separated by up to $T$ time steps. Further, as the order $T$ can theoretically be arbitrarily large, this assumption is relatively mild, albeit, in practice, an upper bound on the order $T$ is often dictated by computational complexity concerns. 

Throughout the paper, we denote by $\bm{\Xi} \subseteq \mathbb{R}^{d}$, with $d = (n+p)(T+1)$, the support of the unknown probability distribution $\mathbb{P}$, and we make the following assumption.
\begin{assumption}\label{ass:support_is_polyhedron}
    The support set $\bm{\Xi}$ is a full-dimensional polytope, that is, $\bm{\Xi}$ contains a $d$-dimensional ball with strictly positive radius.
\end{assumption}
We mainly focus on the case where $\bm{\Xi}$ is a compact polyhedron. Nevertheless, as we will highlight in the following, our results naturally extend to the most studied case $\bm{\Xi} = \mathbb{R}^{d}$.

\begin{remark}
    Reconstructing $\bm{w}_{T}^{(i)}$ and $\bm{v}_{T}^{(i)}$ online, that is, given the corresponding input and output signals $(\bm{u}_T^{(i)}, \bm{y}_T^{(i)})$ only, is in general not possible. Still, the samples in \eqref{eq:training_samples} can be reconstructed from a series of offline experiments conducted in a laboratory environment, where the availability of additional sensors allows measuring the entire state trajectory $\bm{x}_T^{(i)}$ of the system. Alternatively, if $\bm{w}$ and $\bm{v}$ represent the effect of complex physical phenomena, e.g., wind gusts and turbulences, and sensor inaccuracy, respectively, the samples in \eqref{eq:training_samples} can also be generated using high-fidelity simulators.
\end{remark}

\subsection{Control objectives, policies, and uncertainty propagation}
\label{subsec_control_objectives_and_uncertainty}
We consider the problem of designing a stabilizing feedback policy that retains probabilistic safety and performance guarantees over an infinite horizon. Specifically, given $D \succ 0$, we measure the control cost that a policy $\bm{u} = \bm{\pi}(\bm{y})$ incurs whenever the joint disturbance sequence $\bm{\xi}$ realizes as:
\begin{equation*}
    J(\bm{\pi}, \bm{\xi}) = \lim_{T^\prime \to \infty} ~ \frac{1}{T^\prime} \sum_{t=0}^{T^\prime} 
    \begin{bmatrix}
        x_t^\top & u_t^\top    
    \end{bmatrix} D
    \begin{bmatrix}
        x_t \\ u_t
    \end{bmatrix}\,,
\end{equation*}
and we define polytopic safe sets $\mathcal{X} \subseteq \mathbb{R}^{n}$ and $\mathcal{U} \subseteq \mathbb{R}^{m}$ for the system state and input signals, respectively, as:
\begin{align*}
        \mathcal{X} &= \lt\{\! x \in \mathbb{R}^{n} \!: g_x(x) \!=\! \max_{j \in [J_x]} ~ G_{xj}^\top x - g_{xj} \leq 0\,, ~ J_x \!\in \mathbb{N}\rt\}\!, \\
        \mathcal{U} &= \lt\{\! u \in \mathbb{R}^{m} \!: g_u(u) \!=\! \max_{j \in [J_u]} ~ G_{uj}^\top u - g_{uj} \leq 0\,, ~ J_u \!\in \mathbb{N}\rt\}\!,
\end{align*}
where $[J_x]$ denotes the set $\{1,\dots,J_x\} \subset \bb N$ and similarly for $[J_u]$.
Then, given a safety parameter $\gamma \in (0, 1)$ to control the level of tolerable constraint violations, we formulate the following stochastic optimization problem:
\begin{subequations}
\label{eq:original_stochastic_opt_prob}
    \begin{align}
    & ~ \bm{\pi}^\star = \argmin_{\bm{\pi}} ~ \mathbb{E}_{\mathbb{P}} \left[ J(\bm{\pi}, \bm{\xi}) \right] 
    \label{eq:original_stochastic_opt_objective}
    \\
    \label{eq:original_stochastic_opt_constraints}
    & \st ~ 
    \text{CVaR}_{\gamma}^{\mathbb{P}} (\max\{g_x(x_t(\bm{\xi})), g_u(u_t(\bm{\xi}))\}) \!\leq\! 0\,,\!\!
\end{align}
\end{subequations}
where CVaR constraints are defined according to
\begin{equation}
    \label{eq:cvar_constraint_definition}
    \text{CVaR}_{\gamma}^{\mathbb{P}} (g(\bm{\xi})) = \inf_{\tau \in \mathbb{R}} ~ \tau + \frac{1}{\gamma} \mathbb{E}_{\mathbb{P}}[\max\{g(\bm{\xi}) - \tau, 0\}]\,,
\end{equation}
for any measurable function $g: \mathbb{R}^d \to \mathbb{R}$. We note that, besides implying that $\mathbb{P}[x_t \in \mathcal{X}\,, u_t \in \mathcal{U}] \geq 1 - \gamma$, \eqref{eq:original_stochastic_opt_constraints} also accounts for the expected amount of constraint violation in the $\gamma$ percent of cases where any such violation occurs \cite{van2015distributionally}. As such, the CVaR formulation reflects the observation that, in most control applications, severe breaches of the safety constraints often have far more adverse consequences than mild violations. 

As the probability distribution $\mathbb{P}$ is unknown, \eqref{eq:original_stochastic_opt_prob} cannot be addressed directly; instead we rely on a series of robust over-approximations. First, we construct the empirical probability distribution
\begin{equation}
    \label{eq:empirical_center_distribution}
    \widehat{\mathbb{P}} = \frac{1}{N} \sum_{i = 1}^N \delta_{\bm{\xi}_T^{(i)}}\,,
\end{equation}
where $\delta_{\bm{\xi}_T^{(i)}}$ denotes the Dirac delta distribution at the sample $\bm{\xi}_T^{(i)}$. 
%
%
It is well-known that solving \eqref{eq:original_stochastic_opt_prob} upon naively replacing $\mathbb{P}$ with $\widehat{\mathbb{P}}$ may lead to decisions that are unsafe or exhibit poor out-of-sample performance. Indeed, the optimization process often amplifies any estimation error in $\widehat{\mathbb{P}}$ – a phenomenon often termed the \emph{optimizer's curse} in the decision analysis jargon. To immunize against errors in $\widehat{\mathbb{P}}$, we replace the nominal objective \eqref{eq:original_stochastic_opt_objective} with the minimization of the worst-case expected loss over the set of probability distributions $\mathbb{B}_\epsilon\big(\widehat{\mathbb{P}}\big) \subseteq \mathcal{P}(\bm{\Xi})$ that are supported on $\bm{\Xi}$ and are sufficiently close to the empirical estimate $\widehat{\mathbb{P}}$. More formally, we define 
\begin{equation}
    \label{eq:ambiguity_set_definition}
    \mathbb{B}_\epsilon\big(\widehat{\mathbb{P}}\big) = \big\{\mathbb{Q} \in \mathcal{P}(\bm{\Xi}): W\big(\widehat{\mathbb{P}}, \mathbb{Q}\big) \leq \sqrt{\epsilon} \big\}\,,
\end{equation}
where $\epsilon \geq 0$ is the radius of the ambiguity set $\mathbb{B}_\epsilon\big(\widehat{\mathbb{P}}\big)$, and $W\big(\widehat{\mathbb{P}}, \mathbb{Q}\big)$ is the type-2 Wasserstein distance between $\widehat{\mathbb{P}}$ and $\mathbb{Q}$, i.e.,
\begin{equation}
    \label{eq:wasserstein_distance_definition}
    W^2\big(\widehat{\mathbb{P}}, \mathbb{Q}\big) = \inf_{\pi \in \Pi} \int_{\bm{\Xi}^2} \norm{\bm{\xi} - \bm{\xi}^\prime}_2^2 ~ \pi(\bm{\text{d}\xi}, \text{d}\bm{\xi}^\prime)\,,
\end{equation}
where $\Pi$ denotes the set of joint probability distributions of $\bm{\xi}$ and $\bm{\xi}^\prime$ with marginal distributions $\widehat{\mathbb{P}}$ and $\mathbb{Q}$, respectively \cite{mohajerin2018data, kuhn2019wasserstein}. In \eqref{eq:wasserstein_distance_definition}, the decision variable $\pi$ encodes a coupling or transportation plan for moving a mass distribution described by $\widehat{\mathbb{P}}$ to a distribution described by $\mathbb{Q}$. Thus, $\mathbb{B}_\epsilon\big(\widehat{\mathbb{P}}\big)$ can be interpreted as the set of distributions onto which $\widehat{\mathbb{P}}$ can be reshaped at a cost of at most $\epsilon$, where the cost of moving a unit probability mass from $\bm{\xi}$ to $\bm{\xi}^\prime$ is given by $\norm{\bm{\xi} - \bm{\xi}^\prime}_2^2$. We remark that safeguarding against the worst-case distribution in \eqref{eq:ambiguity_set_definition} mitigates the optimizer's curse and, if $\epsilon$ is appropriately tuned, yields a solution that retains finite-samples probabilistic guarantees in terms of out-of-samples control cost and constraint satisfaction. Indeed, for any $\beta > 0$, if $\mathbb{P}$ is light-tailed and the radius $\epsilon$ is a sublinearly growing function of $\frac{\log(1/\beta)}{N}$, then results from measure concentration theory ensure that $\mathbb{P}$ lies inside the ambiguity set \eqref{eq:ambiguity_set_definition} with confidence $1 - \beta$, see, \cite[Theorem~2]{fournier2015rate} and \cite[Theorem~18]{kuhn2019wasserstein}. 


In principle one could address \eqref{eq:original_stochastic_opt_prob} through dynamic programming. since dynamic programming solutions are generally computationally intractable, we restrict our attention to policies $\bm{\pi} \in \bm{\Pi}_{\operatorname{L}}$ that are linear in the past observations $\bm{y}$, that is, $\bm{u} = \bm{\pi}(\bm{y}) = \bm{K}(z)\bm{y}$ for some real-rational proper transfer function $\bm{K}(z)$. Besides computational advantages, our choice is supported by recent advances in DR control, which show that linear policies are globally optimal for a generalization of the classical finite-horizon unconstrained LQG problem, with noise distributions constrained to a Wasserstein ambiguity set \eqref{eq:ambiguity_set_definition}, centered at a nominal Gaussian distribution \cite{taskesen2024distributionally}. 

With these two approximations, our problem of interest becomes:
\begin{subequations}
\label{eq:original_dro_opt_prob}
\begin{align}\label{eq:original_dro_opt_prob_a}
    & ~\inf_{\bm{\pi} \in \bm{\Pi}_{\operatorname{L}}} ~ \sup_{\mathbb{Q} \in \mathbb{B}_{\epsilon}\big(\widehat{\mathbb{P}}\big)} ~ \mathbb{E}_{\mathbb{Q}} \left[ J(\bm{\pi}, \bm{\xi}) \right]\\
    & \st ~ \sup_{\mathbb{Q} \in \mathbb{B}_{\epsilon}\big(\widehat{\mathbb{P}}\big)} ~ \text{CVaR}_{\gamma}^{\mathbb{Q}} (g_t(\bm{\xi})) \leq 0\,, ~ \forall t \in \mathbb{N}\,,
    \label{eq:original_dro_opt_prob_b}
\end{align}
\end{subequations}
where $g_t(\bm{\xi}) = \max\{g_{x}(x_t(\bm{\xi})), g_u(u_t(\bm{\xi}))\}$ for compactness. Note that the worst-case distributions in \eqref{eq:original_dro_opt_prob_a} and \eqref{eq:original_dro_opt_prob_b} may not coincide. Since in practice the uncertainty distribution is unique, this makes \eqref{eq:original_dro_opt_prob} potentially conservative, but is necessary to ensure safety for all distributions in $\mathbb{B}_\epsilon\big(\widehat{\mathbb{P}}\big)$ and not only the one maximizing the expected control cost.
%
%

\subsection{Expressivity of the problem formulation and related work}
The solution to the DRO problem \eqref{eq:original_dro_opt_prob} depends on the radius $\epsilon$ defining \eqref{eq:ambiguity_set_definition}. In particular, we argue that \eqref{eq:original_dro_opt_prob} without constraints generalizes classical $\mathcal{H}_2$ and $\mc H_\infty$ control problems, which correspond to the limit cases of $\epsilon$ approaching $0$ and $\infty$, respectively.

If $\epsilon = 0$, the Wasserstein ball $\mathbb{B}_{\epsilon}\big(\widehat{\mathbb{P}}\big)$ reduces to the singleton $\big\{\widehat{\mathbb{P}}\big\}$, and the supremum disappears, leading to a Monte-Carlo-based control design problem \cite{badings2022sampling, blackmore2010probabilistic}. Moreover, in the absence of constraints, because $J(\bm{\pi}, \bm{\xi})$ is quadratic, the resulting optimal policy is the LQG controller designed for $\bb P_{\mc N} = \mc N(\bb E_{\xi \sim \widehat{\bb P}}[\xi], \operatorname{var}_{\xi \sim \widehat{\bb P}}[\xi])$ \cite{hassibi1999indefinite}. Indeed, since both the system dynamics and the control policy are linear
\begin{align*}
    \mathbb{E}_{\mathbb{P}_{\mc N}} \left[ J(\bm{\pi}, \bm{\xi}) \right] = \mathbb{E}_{\widehat{\mathbb{P}}} \left[ J(\bm{\pi}, \bm{\xi}) \right],
\end{align*}
since both reduce to the same transformation of the first and second moments, which means that the $\argmin_\pi$ of both expectations is also the same.


If $\epsilon$ is very large and $\bm{\Xi}$ is compact, \eqref{eq:original_dro_opt_prob} can also be seen as a generalization of $\mathcal{H}_\infty$ synthesis methods \cite{hassibi1999indefinite, zhou1998essentials}. In fact, in the limit case of $\epsilon \to \infty$ and no matter how $\widehat{\mathbb{P}}$ is constructed, \eqref{eq:ambiguity_set_definition} contains all probability distributions $\mathcal{P}(\bm{\Xi})$ supported on $\bm{\Xi}$, including the degenerate distribution taking value at the most-averse $\bm{\xi}$ almost surely. 

Intermediate values of $\epsilon$ instead yield solutions that leverage the observations \eqref{eq:training_samples} to trade-off robustness to adversarial perturbations or distribution shifts against performance under distributions in a neighborhood of $\widehat{\mathbb{P}}$. 

We conclude this section by remarking that, unlike \cite{taskesen2024distributionally}, we do not assume that the center $\widehat{\mathbb{P}}$ of the ambiguity set is Gaussian, and instead use the empirical estimate \eqref{eq:empirical_center_distribution} to provide greater design flexibility. In fact, if $\mathbb{P}$ is, e.g., bimodal, then the Wasserstein distance between $\mathbb{P}$ and its closest Gaussian distribution $\mathbb{G}$ will generally be larger than the Wasserstein distance between $\mathbb{P}$ and its empirical estimate $\widehat{\mathbb{P}}$. In turn, this implies that a larger radius $\epsilon$ needs to be used to ensure that $\mathbb{P} \in \mathbb{B}_{\epsilon}(\mathbb{G})$ with high probability, leading to a more conservative design. We will return on this point in Section~\ref{sec:numerical}, where we present numerical simulations to support the usage of \eqref{eq:empirical_center_distribution} as center distribution for the Wasserstein ambiguity set \eqref{eq:ambiguity_set_definition} to improve closed-loop performance.

\section{Background}
In this section, we recall useful technical preliminaries, and we discuss the design assumptions that will allow us to compute an approximate solution to \eqref{eq:original_dro_opt_prob} through convex programming. We start by reviewing the system-level approach to controller synthesis \cite{wang2019system}, and then present recent duality results from the DRO literature \cite{shafieezadeh2023new}.

\subsection{System-level synthesis}
The System-Level Synthesis (SLS) framework provides a convex parameterization of the non-convex set of internally stabilizing controllers $\bm{K}(z)$, allowing one to reformulate many control problems as optimization over the closed-loop responses $\bm{\Phi}_{xw}(z), \bm{\Phi}_{xv}(z), \bm{\Phi}_{uw}(z)$ and $\bm{\Phi}_{uv}(z)$ that map $\bm{w}$ and $\bm{v}$ to $\bm{x}$ and $\bm{u}$. To define these maps, we first combine the linear output feedback policy $\bm{u} = \bm{K}(z) \bm{y}$ with the $z$-transform of the state dynamics in \eqref{eq:system_dynamics_definition} to obtain: 
\begin{equation*}
    (z I - (A + B \bm{K}(z) C)) \bm{x} = \bm{w} + B \bm{K}(z) \bm{v}\,. 
\end{equation*}
Then, since the transfer matrix $(z I - (A + B \bm{K}(z) C))$ is invertible for any proper controller $\bm{K}(z)$, we have
\begin{align*}
    \begin{bmatrix}
        \bm{x}\\
        \bm{u}
    \end{bmatrix}
    &= 
    \begin{bmatrix}
        \bm{\Phi}_{xw}(z) & \bm{\Phi}_{xv}(z)\\
        \bm{\Phi}_{uw}(z) & \bm{\Phi}_{uv}(z)
    \end{bmatrix}
    \begin{bmatrix}
        \bm{w}\\
        \bm{v}
    \end{bmatrix}
    =
    \bm{\Phi}_{\xi}(z) \bm{\xi}\,,\\
    &=
    \begin{bmatrix}
        (zI -  (A + B \bm{K}(z) C))^{-1} & \bm{\Phi}_{xw}(z) B \bm{K}(z)\\
        \bm{K}(z) C \bm{\Phi}_{xw}(z) & \bm{\Phi}_{uw}(z) + z\bm{K}(z)
    \end{bmatrix}
    \bm{\xi}\,.
\end{align*}
In particular, we note that causality of $\bm{K}(z)$ implies causality of $\bm{\Phi}_{uv}$ and strict causality of $\bm{\Phi}_{xw}, \bm{\Phi}_{xv}$ and $\bm{\Phi}_{uw}$. Further, one can show that the affine subspace defined by the achievability constraints
\begin{subequations}
\label{eq:achievability_constraints}
\begin{align}
    \begin{bmatrix}
        zI - A & -B 
    \end{bmatrix}
    &\bm{\Phi}_{\xi}(z) =
    \begin{bmatrix}
        I & 0
    \end{bmatrix}\,,\\
    &\bm{\Phi}_{\xi}(z)
    \begin{bmatrix}
        zI - A\\
        -C
    \end{bmatrix} =
    \begin{bmatrix}
        I \\
        0
    \end{bmatrix}\,,
\end{align}
\end{subequations}
characterizes all and only the system responses $\bm{\Phi}_{\xi}(z)$ that are achievable by an internally stabilizing controller $\bm{K}(z)$ \cite{wang2019system}. Despite the fact that  \eqref{eq:achievability_constraints} defines a convex feasible set, minimizing a given convex objective with respect to the closed-loop transfer matrix $\bm{\Phi}_{\xi}(z) = \sum_{k = 0}^{\infty} \Phi(k) z^{-k}$ proves challenging, as the resulting optimization problem remains infinite dimensional. Following \cite{wang2019system, anderson2019system}, we rely on a FIR approximation of $\bm{\Phi}_{\xi}(z)$ to recover tractability, i.e., we restrict our attention to the truncated system response $\bm{\Phi}_{\xi}^{T}(z) = \sum_{k = 0}^{T} \Phi(k) z^{-k}$. We remark that controllability and observability of \eqref{eq:system_dynamics_definition} ensure that \eqref{eq:achievability_constraints} admits a FIR solution \cite[Theorem~4]{wang2019system}. At the same time, since $\bm{\Phi}_{\xi}(z)$ represents a stable map, the effect of this FIR approximation becomes negligible if $T$ is sufficiently large; for the case of LQR, for instance, it was shown that the performance degradation relative to the solution to the infinite-horizon problem decays exponentially with $T$, see \cite[Section~5]{dean2020sample}.

In line with the FIR approximation, we let:
\begin{align*}
    \bm \Phi_x &= [\Phi_{xw}(T), \dots, \Phi_{xw}(0),\Phi_{xv}(T), \dots, \Phi_{xv}(0)]\,,\\
    \bm \Phi_u &= [\Phi_{uw}(T), \dots, \Phi_{uw}(0),\Phi_{uv}(T), \dots, \Phi_{uv}(0)]\,,
\end{align*}
and define $\bm \Phi = [\bm \Phi_x^\top, \bm \Phi_u^\top]^\top$ for compactness. With this notation in place, for any $t \geq T$, we have that:
\begin{equation}
    \label{eq:FIR_xt_and_ut_from_noise_sequence}
    x_t = \bm \Phi_x \bm \xi_{t-T:t}\,, ~
    u_t = \bm \Phi_u \bm \xi_{t-T:t}\,,
\end{equation}
where $\bm \xi_{t-T:t} = [w_{t-T}^\top, \dots, w_{t}^\top, v_{t-T}^\top, \dots, v_{t}^\top]^\top$ collects the last $T+1$ realizations of the process and measurement noises. The closed loop maps $\bm \Phi_x$ and $\bm \Phi_u$ can be utilized to compute a practical implementation of the controller $\bm K(z)$ as shown in Appendix \ref{app:controller_implementation}.

\subsection{A stationary control problem}

As we consider an infinite horizon control problem, we focus on the steady state behavior of the system, and we are instead less interested in the transient behavior, compare also \cite{van2015distributionally}. Motivated by this and to take full advantage of the stationarity properties of $\xi_t$ in Assumption \ref{ass:noise_follows_stationary_process}, we focus on designing an optimal safe controller to operate the system for $t \geq T$ only. In this setting, we proceed to show that the distributionally robust worst-case control cost and CVaR constraints admit finite-dimensional representations
\begin{assumption}\label{ass:system_initialized}
    The system is initialized by an external controller with $x_0, \dots, x_{T-1} \in \mathcal{X}$ and $u_0, \dots, u_{T-1} \in \mathcal{U}$.
\end{assumption}

%
We therefore redefine the optimization cost $J$ in \eqref{eq:original_dro_opt_prob_a} as \begin{equation*}
    J_T(\bm{\pi}(\bm \Phi), \bm{\xi}) = \lim_{T^\prime \to \infty} \! \frac{1}{T^\prime \!-\! T} \! \sum_{t=T}^{T^\prime} 
    \bm \xi_{t-T:t}^\top \bm \Phi^\top D \bm \Phi \bm \xi_{t-T:t}\,.
\end{equation*}
Note that under Assumption \ref{ass:noise_follows_stationary_process}),the uncertainty $\bm{\xi}_{t-T:t}$ follows the same distribution for all $t$. Hence, $J_T$ satisfies
\begin{align}
    \bb E_{\bb Q} &[J(\bm \pi(  \bm \Phi), \bm\xi)]
    \nonumber \\
    &= 
    \lim_{T^\prime \to \infty} \!
    \bb E_{\!\!\!\substack{\xi_{0:T} \sim \bb Q
    \vspace{-4pt}
    \\
    \vdots
    \\
    \xi_{T^\prime-T:T^\prime} \sim \bb Q}}
    \left[ \frac{1}{T^\prime \!-\! T} \! 
    \sum_{t=T}^{T^\prime} \bm \xi_{t-T:t}^\top \bm \Phi^\top D \bm \Phi \bm \xi_{t-T:t}\right]\!\!,
    \nonumber \\ \label{eq:quad_expectation_risk_def}
    &= \bb E_{\xi_T \sim \bb Q} [\bm \xi_T^\top \bm \Phi^\top D \bm \Phi \bm \xi_T]\,.
\end{align}




In light of this, the problem statement \eqref{eq:original_dro_opt_prob}  for DRInC synthesis can be reformulated as finding the optimal FIR map $\bm\Phi^\star$ of length $T+1$ given by
 \begin{align}\label{eq_risk_def}
    \bm{\Phi}^\star = \argmin_{\bm{\Phi} \mt{ achievable}}
    \sup_{
    \mathbb{Q} \in \mathbb{B}_{\epsilon}\big(\widehat{\mathbb{P}}\big) 
    }
    \bb E_{\bm \xi_{T} \sim \mathbb{Q}}
    [ \bm \xi_{T}^\top \bm \Phi^\top D \bm \Phi \bm \xi_{T}] \,,
\end{align}
while satisfying the achievability constraints \eqref{eq:achievability_constraints} as well as conditional value-at-risk constraints
\begin{align} 
\label{eq:conditional_value_at_risk_safety_constraints_inf_horizon_both}
    \!\!&\sup_{
    \mathbb{Q} \in \mathbb{B}_{\epsilon}\big(\widehat{\mathbb{P}}\big) 
    } \! \text{CVaR}_{\gamma}^{{\bm\xi_T} \sim \bb Q}({G}_{j}^\top \bm \Phi \bm \xi_T - g_{j}) \leq 0, \forall j\! \in [J],\!
\end{align}
where $J = J_x + J_u$ and $[J] = \{1,\dots,J\}$ enumerates all the constraints on $[x^\top, u^\top]$, defined by 
\begin{align*}
    G = \lt[\matb G_x & 0 \\ 0 & G_u \mate\rt], \; g = \lt[\matb g_x \\ g_u \mate\rt]\!.
\end{align*}


We highlight that while \eqref{eq:original_dro_opt_prob_a} is an infimum problem, the minimum in \eqref{eq_risk_def} is attained. Indeed, by Assumption \ref{ass:support_is_polyhedron}, there always exists a distribution $\widehat{\bb Q}$ such that $\bb E_{\widehat{\bb Q}} [J(\bm \pi( \bm \Phi), \bm\xi)] $ is strongly convex in $\bm\Phi$ (e.g., an empirical distribution containing samples that form a basis for $\bb R^d$). Moreover, since $\bb E_{\widehat{\bb Q}} [J(\bm \pi( \bm \Phi), \bm\xi)] \leq \sup_{ \mathbb{Q} \in \mathbb{B}_{\epsilon}\big(\widehat{\mathbb{P}}\big) } \bb E_{\bb Q} [J(\bm \pi( \bm \Phi), \bm\xi) ]$ by definition, the supremum in \eqref{eq_risk_def} is strongly convex and the minimizer $\bm \Phi^\star$ is attainable. Instead, as both the transport and control costs grow quadratically, $\sup_{\mathbb{Q} \in \mathbb{B}_{\epsilon}\big(\widehat{\mathbb{P}}\big)} ~ \mathbb{E}_{\mathbb{Q}} \left[ J(\bm{\pi}(\bm \Phi), \bm{\xi}) \right]$ may be unattainable \cite{shafieezadeh2023new}\footnote{The ratio between the growth rates of the loss function and the transport cost is crucial in DRO problems. If the control cost grows faster than the transport cost, the adversary can make the control cost diverge by moving an infinitesimal amount of mass very far away from the empirical distribution. Conversely, if the control cost grows slower, there is always a point at which the amount of transported mass is too small to further increase the risk and the supremum is attained. This is the case for the constraints, as their cost grows linearly.}. In what follows, we use the recent advances in DRO theory presented in \cite{shafieezadeh2023new} to reformulate the control design problem as a finite-dimensional and tractable problem.

\subsection{Strong duality for DRO of piecewise linear objectives}\label{subsec_present_prop_212}
The minimization \eqref{eq_risk_def} subject to \eqref{eq:conditional_value_at_risk_safety_constraints_inf_horizon_both} is infinite-dimensional and cannot be directly approached. The next proposition, which serves as a starting point for our derivations in Section~\ref{sec:main_results}, shows how DRO of piecewise linear objectives can be recast as a sample-based finite-dimensional convex program.

\begin{proposition}\label{prop:cvar_original_deterministic_constraint_bounded_support}
Let $a_j \in \bb R^d$ and $b_j \in \bb R$, for $j  = 1, \dots, J$, characterize the piece-wise linear cost $\max_{j \in [J]}  a_j^\top \bm\xi_T + b_j$, and let $H$ and $h$ characterize the support $\bm\Xi$ as $\{\bm{\xi} \in \mathbb{R}^{d} : {H} \bm{\xi} \leq {h}\}$. If Assumption \ref{ass:support_is_polyhedron} holds and $\epsilon > 0$, then the risk:
\begin{align}\label{eq:prop_dro_general_risk_def}
    \sup_{\mathbb{Q} \in \mathbb{B}_{\epsilon}\big(\widehat{\mathbb{P}}\big)} \bb E_{\xi_T \sim \bb Q} \left[\max_{j \in [J]}  a_j^\top \bm\xi_T + b_j\right]\,,
\end{align}
can be equivalently computed as:
\begin{subequations}\label{eq:prop_dro_general_risk_dual}
    \begin{align}
    \label{eq:prop_dro_general_risk_dual_cost}
    &\inf_{s^{(i)},
    \lambda \geq 0, \kappa_{ij} \geq 0}  
    \lambda \epsilon + \frac{1}{N} \!\sum_{i \in [N]} s^{(i)}\,,\; \st
    \\ \label{eq:eq:prop_dro_general_risk_dual_inequality}
    & s^{(i)} \geq  b_j \!+\! \frac{\|a_j\|_2^2}{4\lambda}  + a_j^\top \bm \xi_T^{(i)}
    \\ \nonumber
    &\quad\quad\; +\! \frac{1}{4\lambda} \kappa_{ij}^\top H H^\top \kappa_{ij} \!-\! \frac{1}{2\lambda} a_{j}^\top H^\top \kappa_{ij} + \big(h - H \bm \xi_T^{(i)} \big)^{\!\!\top} \kappa_{ij},
    \end{align}
\end{subequations}
    for all $i = 1, \dots, N$ and $j = 1, \dots, J$.
\end{proposition}
\begin{proof}
    This proposition is a special case of \cite[Proposition~2.12]{shafieezadeh2023new}. For the sake of clarity, we report detailed derivations in Appendix \ref{app:proof_prop_bounded_support}.
\end{proof}
Proposition \ref{prop:cvar_original_deterministic_constraint_bounded_support} uses strong duality to establish an equivalence between \eqref{eq:prop_dro_general_risk_dual} and \eqref{eq:prop_dro_general_risk_def}. In particular, the decision variables $\lambda \in \bb R$ and $\kappa_{ij}\in \bb R^{n_H}$, where $n_H$ is the number of rows in $H$, correspond to the Lagrange multipliers associated with the constraints $\bb Q \in \bb B_\epsilon\big(\widehat{\mathbb{P}}\big)$ and $\bm \xi_T \in \bm \Xi$, respectively. The optimal value of $\lambda$ can thus be interpreted as the shadow cost of robustification, i.e., the amount by which the risk $\bb E_{\xi_T \sim \bb Q} [\max_{j \in [J]}  a_j^\top \bm\xi_T + b_j]$ increases for each unit increase of $\epsilon$, written mathematically as \eqref{eq:shadow_cost_def} below. The variables $s^{(i)} \in \bb R$ instead represent epigraph variables. 

\section{Main Results}
\label{sec:main_results}
In this section, we present our main results. Motivated by the observation that the operational costs of engineering applications usually relate to energy consumption and are thus often modeled using quadratic functions, we proceed to extend the results of Proposition \ref{prop:cvar_original_deterministic_constraint_bounded_support} beyond piecewise linear objectives. 

\subsection{Non-convexity challenges}\label{subsec_quad_dro}
Note that the strong duality results of \cite[Proposition~2.12]{shafieezadeh2023new} do not directly apply to \eqref{eq_risk_def}, as the objective $J(\bm{\pi}(\bm \Phi), \bm{\xi})$ is not piece-wise concave. Therefore, unless $\bm{\Xi}$ equals the entire space $\mathbb{R}^d$, an extension of current state-of-the-art results in DRO is required to minimize a risk of the form
\begin{align}\label{eq:prop_dro_general_quad_risk_def}
    \mc R_\epsilon(Q) := \sup_{\mathbb{Q} \in \mathbb{B}_{\epsilon}\big(\widehat{\mathbb{P}}\big)}  \bb E_{\bm\xi_T \sim \bb Q} ~  [\bm\xi_T^\top Q \bm\xi_T]\,,
\end{align}
where $Q \succeq 0$. We start by observing that if the loss is not concave with respect to $\bm\xi_T$, then the optimization problem in \eqref{eq:prop_dro_general_quad_risk_def}  may not be convex. In fact, while \cite{kuhn2019wasserstein} shows that there is a hidden convexity when $\bm \Xi = \bb R^d$, this result does not hold in general. To illustrate this point, consider for example the situation drawn in Fig. \ref{fig:nonconvex_adversarial_play}, where $\bm{\Xi} \subset \mathbb{R}$. One can observe that if the constraint $\bb Q \in \mathbb{B}_\epsilon(\delta)$ is active, then the problem \eqref{eq:prop_dro_general_quad_risk_def} amounts to a Quadratically Constrained Quadratic Program (QCQP), which admits a tight convex relaxation as a Semi-Definite Program (SDP) \cite{boyd2004convex}. Conversely, however, when the constraint $\bb Q \in \mathbb{B}_\epsilon(\delta)$ is not active, \eqref{eq:prop_dro_general_quad_risk_def} amounts to maximize a convex Quadratic Program (QP), which is not convex. 
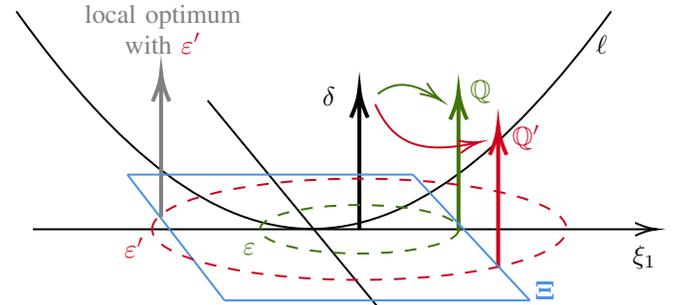
\begin{figure}[h]
    \centering
    \tikzset{every picture/.style={line width=0.75pt}} 

\begin{tikzpicture}[x=0.75pt,y=0.75pt,yscale=-1,xscale=1]

\draw    (35.33,150) -- (348,150) ;
\draw [shift={(350,150)}, rotate = 180] [color={rgb, 255:red, 0; green, 0; blue, 0 }  ][line width=0.75]    (10.93,-3.29) .. controls (6.95,-1.4) and (3.31,-0.3) .. (0,0) .. controls (3.31,0.3) and (6.95,1.4) .. (10.93,3.29)   ;
\draw [color={rgb, 255:red, 74; green, 144; blue, 226 }  ,draw opacity=1 ][line width=0.75]    (132,185.91) -- (83,122.41) ;
\draw [color={rgb, 255:red, 0; green, 0; blue, 0 }  ,draw opacity=1 ][line width=1.5]    (200,150) -- (200,82.41) ;
\draw [shift={(200,79.41)}, rotate = 90] [color={rgb, 255:red, 0; green, 0; blue, 0 }  ,draw opacity=1 ][line width=1.5]    (14.21,-4.28) .. controls (9.04,-1.82) and (4.3,-0.39) .. (0,0) .. controls (4.3,0.39) and (9.04,1.82) .. (14.21,4.28)   ;
\draw    (27,40) .. controls (131,185.41) and (219,187.41) .. (327,40) ;
\draw [color={rgb, 255:red, 65; green, 117; blue, 5 }  ,draw opacity=1 ][line width=1.5]    (250,151) -- (250,83.41) ;
\draw [shift={(250,80.41)}, rotate = 90] [color={rgb, 255:red, 65; green, 117; blue, 5 }  ,draw opacity=1 ][line width=1.5]    (14.21,-4.28) .. controls (9.04,-1.82) and (4.3,-0.39) .. (0,0) .. controls (4.3,0.39) and (9.04,1.82) .. (14.21,4.28)   ;
\draw [color={rgb, 255:red, 208; green, 2; blue, 27 }  ,draw opacity=1 ][line width=1.5]    (270,168.81) -- (270,101.22) ;
\draw [shift={(270,98.22)}, rotate = 90] [color={rgb, 255:red, 208; green, 2; blue, 27 }  ,draw opacity=1 ][line width=1.5]    (14.21,-4.28) .. controls (9.04,-1.82) and (4.3,-0.39) .. (0,0) .. controls (4.3,0.39) and (9.04,1.82) .. (14.21,4.28)   ;
\draw [color={rgb, 255:red, 65; green, 117; blue, 5 }  ,draw opacity=1 ]   (209.33,81.74) .. controls (219.47,73.58) and (229.6,77.38) .. (239.73,84) ;
\draw [shift={(241.33,85.07)}, rotate = 214.24] [color={rgb, 255:red, 65; green, 117; blue, 5 }  ,draw opacity=1 ][line width=0.75]    (10.93,-3.29) .. controls (6.95,-1.4) and (3.31,-0.3) .. (0,0) .. controls (3.31,0.3) and (6.95,1.4) .. (10.93,3.29)   ;
\draw [color={rgb, 255:red, 208; green, 2; blue, 27 }  ,draw opacity=1 ]   (208,86.91) .. controls (216.82,106.67) and (230.44,112.67) .. (259.22,106.31) ;
\draw [shift={(261,105.91)}, rotate = 166.87] [color={rgb, 255:red, 208; green, 2; blue, 27 }  ,draw opacity=1 ][line width=0.75]    (10.93,-3.29) .. controls (6.95,-1.4) and (3.31,-0.3) .. (0,0) .. controls (3.31,0.3) and (6.95,1.4) .. (10.93,3.29)   ;
\draw [color={rgb, 255:red, 128; green, 128; blue, 128 }  ,draw opacity=1 ][line width=1.5]    (100,143.81) -- (100,74.22) ;
\draw [shift={(100,71.22)}, rotate = 90] [color={rgb, 255:red, 128; green, 128; blue, 128 }  ,draw opacity=1 ][line width=1.5]    (14.21,-4.28) .. controls (9.04,-1.82) and (4.3,-0.39) .. (0,0) .. controls (4.3,0.39) and (9.04,1.82) .. (14.21,4.28)   ;
\draw    (229.23,213.46) -- (123.5,85) ;
\draw [shift={(230.5,215)}, rotate = 230.54] [color={rgb, 255:red, 0; green, 0; blue, 0 }  ][line width=0.75]    (10.93,-3.29) .. controls (6.95,-1.4) and (3.31,-0.3) .. (0,0) .. controls (3.31,0.3) and (6.95,1.4) .. (10.93,3.29)   ;
\draw  [color={rgb, 255:red, 208; green, 2; blue, 27 }  ,draw opacity=1 ][dash pattern={on 4.5pt off 4.5pt}] (95.5,150) .. controls (95.5,135.89) and (142.29,124.45) .. (200,124.45) .. controls (257.71,124.45) and (304.5,135.89) .. (304.5,150) .. controls (304.5,164.11) and (257.71,175.55) .. (200,175.55) .. controls (142.29,175.55) and (95.5,164.11) .. (95.5,150) -- cycle ;
\draw  [color={rgb, 255:red, 65; green, 117; blue, 5 }  ,draw opacity=1 ][dash pattern={on 4.5pt off 4.5pt}] (149.88,150) .. controls (149.88,143.23) and (172.32,137.75) .. (200,137.75) .. controls (227.68,137.75) and (250.13,143.23) .. (250.13,150) .. controls (250.13,156.77) and (227.68,162.25) .. (200,162.25) .. controls (172.32,162.25) and (149.88,156.77) .. (149.88,150) -- cycle ;
\draw [color={rgb, 255:red, 74; green, 144; blue, 226 }  ,draw opacity=1 ][line width=0.75]    (286,185.91) -- (227,122.41) ;
\draw [color={rgb, 255:red, 74; green, 144; blue, 226 }  ,draw opacity=1 ][line width=0.75]    (286,185.91) -- (132,185.91) ;
\draw [color={rgb, 255:red, 74; green, 144; blue, 226 }  ,draw opacity=1 ][line width=0.75]    (227,122.41) -- (83,122.41) ;

\draw (287,174.81) node [anchor=north west][inner sep=0.75pt]    {$\textcolor[rgb]{0.29,0.56,0.89}{\mathbf{\Xi }}$};
\draw (139.88,155.4) node [anchor=north west][inner sep=0.75pt]    {$\textcolor[rgb]{0.25,0.46,0.02}{\varepsilon }$};
\draw (80,152.81) node [anchor=north west][inner sep=0.75pt]    {$\textcolor[rgb]{0.82,0.01,0.11}{\varepsilon '}$};
\draw (318,50.4) node [anchor=north west][inner sep=0.75pt]  [color={rgb, 255:red, 0; green, 0; blue, 0 }  ,opacity=1 ]  {$\ell $};
\draw (180,76.73) node [anchor=north west][inner sep=0.75pt]    {$\delta $};
\draw (254.33,73.07) node [anchor=north west][inner sep=0.75pt]    {$\mathbb{\textcolor[rgb]{0.25,0.46,0.02}{Q}}$};
\draw (275.67,96.07) node [anchor=north west][inner sep=0.75pt]    {$\mathbb{\textcolor[rgb]{0.82,0.01,0.11}{Q}}\textcolor[rgb]{0.82,0.01,0.11}{'}$};
\draw (55.67,35) node [anchor=north west][inner sep=0.75pt]   [align=left] {\begin{minipage}[lt]{65.65pt}\setlength\topsep{0pt}
\begin{center}
\textcolor[rgb]{0.5,0.5,0.5}{local optimum}\\\textcolor[rgb]{0.5,0.5,0.5}{with} $\displaystyle \textcolor[rgb]{0.82,0.01,0.11}{\varepsilon '}$
\end{center}

\end{minipage}};
\draw (336.33,156.4) node [anchor=north west][inner sep=0.75pt]    {$\xi _{1}$};
\draw (235,202.4) node [anchor=north west][inner sep=0.75pt]    {$\xi _{2}$};

\end{tikzpicture}
    \vspace{-16pt}
    \caption{Illustration of two worst-case distributions $\bb Q \in \mathbb{B}_\epsilon(\delta)$ and $\bb Q' \in \mathbb{B}_{\epsilon'}(\delta)$ in different Wasserstein balls around the Dirac delta distribution. The support $\bm\Xi$ is represented by the horizontal blue square, and the left-most Dirac distribution represents a local minima in $\mathbb{B}_{\epsilon'}(\delta)$ for the risk $\mc R_\epsilon(Q)$ in \eqref{eq:prop_dro_general_quad_risk_def}.}
    \label{fig:nonconvex_adversarial_play}
\end{figure}

Whether the constraint $\bb Q \in \mathbb{B}_\epsilon(\delta)$ is active or not is indicated by the value taken at the optimum by its Lagrange multiplier $\lambda$, which represents the shadow cost of robustification defined by
\begin{align}\label{eq:shadow_cost_def}
    \lambda^\star(Q) = \frac{d \mc R_\epsilon(Q)}{d\epsilon^+} = \lim_{\Delta\epsilon \rightarrow 0^+} \frac{\mc R_{\epsilon + \Delta\epsilon}(Q) - \mc R_\epsilon(Q)}{\Delta\epsilon} \,,
\end{align}
where existence of the limit from above follows by the monotonicity of $\mathcal{R}_\epsilon(Q)$ with respect to $\epsilon$. The following proposition provides a sufficient condition for the constraint to be active, generalizing the example shown in Fig. \ref{fig:nonconvex_adversarial_play} to $\bb R^d$.

\begin{proposition}\label{prop:lambda_is_big_enough_if_support_is}
Let $\partial \bm{\Xi} = \{\bm \xi : \max\limits_{k\in [n_H]} H_k \bm \xi -  h_k = 0 \}$, denote the boundary of the polytope $\bm{\Xi}$. If 
\begin{align}\label{eq:prop_lambda_is_big_enough_if_support_is}
    \frac{1}{N} \sum_{i \in [N]} ~ \min_{\tilde{\bm{\xi}} \in \partial \bm{\Xi}} ~ \norm{\bm{\xi}_T^{(i)}- \tilde{\bm{\xi}}}^2_2 > \epsilon\,,
\end{align}
that is, if the average squared distance between the samples and the boundary $\partial \bm\Xi$ of  $\bm\Xi$ is strictly greater than epsilon, then the optimal shadow cost of robustification $\lambda^\star(Q)$ in \eqref{eq:shadow_cost_def} is greater than $\lambda_{\max}(Q)$ for any positive semidefinite $Q \in \bb R^{d \times d}$.
\end{proposition}
\begin{proof}
    The proof is given in Appendix \ref{app:lambda_is_big_enough_if_support_is}.
\end{proof}
Proposition \ref{prop:lambda_is_big_enough_if_support_is} shows that $\lambda$ is contingent on the radius $\epsilon$, the support $\bm \Xi$, and the realizations $\bm \xi^{(i)}$. The radius $\epsilon$ is usually small, as the samples should be approximating the real distribution well enough, which means that the condition \eqref{eq:prop_lambda_is_big_enough_if_support_is} is often satisfied. In the next section, we utilize the inequality $\lambda^\star \geq \lambda_{\max}(Q)$ to propose a strong dual formulation for \eqref{eq:prop_dro_general_quad_risk_def}.


\subsection{Tight convex relaxation for DRO of quadratic objectives}\label{subsec_results_dro_theory}
In this section, we present a convex upper bound for \eqref{eq:prop_dro_general_quad_risk_def}, and prove that it becomes tight if $\lambda$ is greater than $\lambda_{\max}(Q)$, the largest eigenvalue of $Q$.

\begin{lemma}\label{coro:quad_dro_prob_is_limit_linear}
Let $Q \in \bb R^{d\times d}$ be a positive definite matrix. Under Assumption \ref{ass:support_is_polyhedron}, if $\epsilon > 0$ and if $\bm \Xi$ is bounded, the risk \eqref{eq:prop_dro_general_quad_risk_def} satisfies
\begin{subequations}\label{eq:coro_dro_general_quad_risk_dual}
\begin{align}
    \label{eq:coro_dro_general_quad_risk_dual_cost}
    \!\!\mc R_\epsilon(Q) \leq &\inf_{\substack{
    s^{(i)},\, \lambda \geq 0, \\ \psi_i \geq \mu_i, \\ \mu_i \geq 0,\, \alpha \geq 0}}  
    \lambda \epsilon + \frac{1}{N} \!\sum_{i \in [N]} s^{(i)}\,,\; \\ \nonumber
    &\st\,, \forall i \in [N]:
    \\ \label{eq:coro_dro_general_quad_risk_dual_inequality}
    &\lt[\matb
    s^{(i)} \!\!-\!  h^{\!\top}\! \psi_i \!+\! \lambda\|\bm\xi_T^{(i)} \! \|_2^2 \!\!\!\!\!\!
    &
    \star
    &
    \!\!\star
    \\
    2 \lambda \bm\xi_T^{(i)} \!\!-\!  H^{\!\top} \!\psi_i
    &
    4(\lambda I \!-\! Q)
    &
    \!\!\star 
    \\
    H^{\!\top} \mu_i
    &
    0
    &
    \!\!4Q
    \mate\rt] \!\! \succeq \! 0, \!\!
    \\ \label{eq:coro_dro_general_quad_risk_dual_equality}
    &
    \lt[\matb
    \alpha & \star
    \\
    H^{\!\top}\!\mu_i & \lambda I - Q
    \mate\rt] \succeq 0.
\end{align}
\end{subequations}
Moreover, \eqref{eq:coro_dro_general_quad_risk_dual_cost} holds with equality if the optimum $\lambda^\star$ of $\lambda$ satisfies $\lambda^\star \geq \lambda_{\max}(Q)$.
\end{lemma}
\vspace{5pt}
\begin{proof}
    This result is obtained by taking the limit of \eqref{eq:prop_dro_general_risk_dual} when the number $J$ of pieces tends to infinity. The detailed derivations are presented in Appendix \ref{app:proof_lem_bounded_support_quad}.
\end{proof}

We stress that our results continue to hold even if $H = 0$ and $h = 0$, that is, if $\bm \Xi = \bb R^d$. In this case, \eqref{eq:coro_dro_general_quad_risk_dual} simplifies substantially.

\begin{corollary}\label{cor_drinc_cost}
Lemma \ref{coro:quad_dro_prob_is_limit_linear} also holds if $\bm \Xi = \bb R^d$ and \eqref{eq:coro_dro_general_quad_risk_dual} simplifies into
\begin{subequations}
\label{eq:kuhn_theorem_11}
    \begin{align}\label{eq:kuhn_theorem_11_cost}
    \mc R_\epsilon(Q) =\; &\inf_{\substack{s^{(i)}, \lambda \geq 0}}
     \lambda \epsilon + \frac{1}{N} \!\sum_{i \in [N]} s^{(i)}\,,
    \\ \label{eq:kuhn_theorem_11_inequality}
    & \st \lt[\matb
    s^{(i)} \!+\! \lambda\|\bm\xi_T^{(i)} \! \|_2^2 \!\!
    &
    \star
    \\
    \lambda \bm\xi_T^{(i)}
    &
    \lambda I - Q
    \mate\rt] \!\! \succeq \! 0. \!\!\!
    \end{align}
\end{subequations}
\end{corollary}
\begin{proof}
    If $\bm\Xi = \bb R^d$, the problem \eqref{eq_risk_def} falls into the assumptions of \cite[Theorem 11]{kuhn2019wasserstein}. Additionally, we observe that, when $H = 0$ and $h = 0$, \eqref{eq:kuhn_theorem_11_inequality} has the same Schur complement as \eqref{eq:coro_dro_general_quad_risk_dual_inequality} and \eqref{eq:coro_dro_general_quad_risk_dual_equality} is always satisfied.
\end{proof}

To understand the effect of having restricted our attention to distributions with bounded support, it is of interest to compare \eqref{eq:coro_dro_general_quad_risk_dual}  with \eqref{eq:kuhn_theorem_11}. In both problems, the presence of the term $\lambda I - Q$ in \eqref{eq:coro_dro_general_quad_risk_dual_inequality} and \eqref{eq:kuhn_theorem_11_inequality} implies that any feasible solution has a shadow cost $\lambda$ greater or equal than $\lambda_{\max}(Q)$. On the other hand, for \eqref{eq:kuhn_theorem_11_inequality} to be feasible, $\lambda$ should be large enough to guarantee $s^{(i)} + \lambda\|\bm\xi_T^{(i)}\|_2^2 \geq 0$, whereas the presence of the additional term $-h^\top \psi_i$ in the top-left entry of \eqref{eq:coro_dro_general_quad_risk_dual_inequality} softens this requirement, demonstrating the helpful contribution of the bounded support.


\subsection{Convex formulation of DRInC design}\label{subsec_results_drinc}

Our results of Section \ref{subsec_results_dro_theory} do not directly allow us to solve \eqref{eq_risk_def}, as \eqref{eq:quad_expectation_risk_def} shows that the control cost depends quadratically on $\bm \Phi$ and $\bm \Phi^\top D \bm \Phi$ may be rank deficient. In this subsection, we mitigate the issues associated with quadratic matrix inequalities by employing a Schur complement. We address rank-deficiency concerns by examining the risk $\mc R_\epsilon (\bm \Phi^\top D \bm \Phi + |\eta| I)$ as its argument approaches singularity, showing that the limit $\eta \rightarrow 0$ remains well-behaved.

\begin{lemma}\label{lem:drinc_cost}
Under Assumption \ref{ass:support_is_polyhedron}, if $\epsilon > 0$ and $\bm\Xi$ is bounded, the optimal closed loop map $\bm \Phi^\star$ in \eqref{eq_risk_def} is given by
\begin{subequations}
\label{eq:lem_quad_cost}
\begin{align}
    \label{eq:lem_quad_cost_cost}
    \bm\Phi^\star = \underset{\bm{\Phi}}{\arg} &\min_{\substack{\bm\Phi \mt{ achievable}, Q}}
    \lim_{\eta \rightarrow 0}
    \mc R_\epsilon(Q + |\eta| I) \,,
    \\ \label{eq:lem_quad_cost_lmi}
    & \st \,
    \!\lt[\matb
    Q & \star
    \\
    D^{\frac{1}{2}} \Phi & I
    \mate\rt]
    \succeq 0 \,.
    %
\end{align}
\end{subequations}
\end{lemma}
\begin{proof}
    The proof can be found in Appendix \ref{app:proof_lem_drinc_cost}
\end{proof}

We continue our derivations by presenting an equivalent convex reformulation of the safety constraints in \eqref{eq:conditional_value_at_risk_safety_constraints_inf_horizon_both}. In particular, in the next proposition, we embed the function $\max\{\cdot - \tau, 0\}$ in \eqref{eq:cvar_constraint_definition} as a $(J+1)^\text{th}$ constraint.

\begin{lemma}\label{lem:drinc_constraints}
Under Assumption \ref{ass:support_is_polyhedron} and if $\epsilon > 0$, the constraints \eqref{eq:conditional_value_at_risk_safety_constraints_inf_horizon_both} can be reformulated as the following convex LMIs
\begin{subequations}
\label{eq:lem_quad_constraints}
\begin{align}
    \label{eq:lem_quad_constraints_cost}
    &  
    \rho \epsilon + \frac{\gamma - 1}{\gamma}\tau + \frac{1}{N} \!\sum_{i \in [N]} \zeta^{(i)} \leq 0 \,, \;\rho \geq 0 \,,
    \\ \label{eq:lem_quad_constraints_inequality}
    & \forall i \in [N] \,, \forall j \in [J+1] : \; \kappa_{ij} \geq 0\,,
    \\ \label{eq:lem_quad_constraints_lmi}
    & \!\!\lt[\matb
    \zeta^{(i)}\! \!-\! \frac{1}{\gamma} \!( G_{j}^\top \bm\Phi \bm \xi_T^{(i)} \!\!-\! g_{j} ) \!+\! ( H \bm \xi_T^{(i)} \!\!-\! h)^{\!\!\top} \!\kappa_{ij} \!\!\!\!\!
    &
    \star
    \\ 
    \bm\Phi^{\!\top} G_{j} \!-\! \gamma H^{\!\top}\! \kappa_{ij}
    &
    4\rho \gamma^2 I
    \mate\rt] \!\!\succeq\! 0,\!
\end{align}
\end{subequations}
where $G_{J+1} = 0$ and $g_{J+1} = -\tau$.
\end{lemma}
\begin{proof}
    The proof can be found in Appendix \ref{app:proof_lem_drinc_constraints}.
\end{proof}

Leveraging Lemmas \ref{coro:quad_dro_prob_is_limit_linear}, \eqref{lem:drinc_cost}, and \eqref{lem:drinc_constraints}, we are now ready to reformulate \eqref{eq_risk_def} subject to \eqref{eq:conditional_value_at_risk_safety_constraints_inf_horizon_both} as SDP.
\begin{theorem}\label{thm:drinc}
Under Assumption \ref{ass:support_is_polyhedron} and if $\epsilon > 0$, the closed loop map given by
\begin{align*}
    \bm\Phi^\star = \argmin_{\bm\Phi \mt{ achievable}}\inf_{\substack{
    Q, s^{(i)}, \zeta^{(i)}, \tau,
    \\
    \lambda \geq 0, \rho \geq 0, \alpha \geq 0,
    \\
    \mu_i \geq 0, \kappa_{ij} \geq 0,
    \\ \psi_i \geq \mu_i}}  
    \hspace{-10pt}&\hspace{10pt}
    \lambda \epsilon + \frac{1}{N} \!\sum_{i \in [N]} s^{(i)}\,,
    \\ \nonumber
    &\st
    \\
    &\eqref{eq:lem_quad_cost_lmi}, \eqref{eq:lem_quad_constraints_cost},
    \\ &\eqref{eq:coro_dro_general_quad_risk_dual_inequality}, 
    \eqref{eq:coro_dro_general_quad_risk_dual_equality}, \hspace{4pt} \forall i \!\in\! [N],
    \\ &\eqref{eq:lem_quad_constraints_lmi},
    \hspace{30pt} \forall i \!\in\! [N], j \!\in\! [J \!+\! 1],
\end{align*}
is stable and satisfies the safety constraints \eqref{eq:conditional_value_at_risk_safety_constraints_inf_horizon_both}. Moreover, it optimizes \eqref{eq_risk_def} if $\bm\Xi$ is bounded and the optimizer $\lambda^\star$ is greater than $\lambda_{\max}\big({\bm\Phi^\star}^\top D \bm\Phi^\star\big)$.
\end{theorem}
\begin{proof}
    We first highlight that $\bm \Phi$ is FIR and therefore stable by definition. Second, the safety constraints \eqref{eq:conditional_value_at_risk_safety_constraints_inf_horizon_both} are equivalent to \eqref{eq:lem_quad_constraints}, as shown in Lemma \ref{lem:drinc_constraints}. Third, consider a closed loop map $\widehat{\bm\Phi}$, which optimizes the expectation of $\bm \xi_{T}^\top \widehat{\bm\Phi}^{\!\top} D \widehat{\bm\Phi} \bm \xi_{T} + |\eta|\|\bm\xi_T\|_2^2$ for $\eta \neq 0$. With $Q = \bm \Phi^\top D \bm \Phi + |\eta| I \succ 0$, Lemma \ref{coro:quad_dro_prob_is_limit_linear} shows that $\mc R_\epsilon(\bm \Phi^{\!\top} D \bm \Phi \bm)$ is tightly upper-bounded by \eqref{eq:coro_dro_general_quad_risk_dual}. Fourth and finally, as shown in Lemma \ref{lem:drinc_cost}, taking the limit $\eta \rightarrow 0$ yields $\widehat{\bm\Phi} \rightarrow \bm\Phi^\star$ from \eqref{eq_risk_def}, which concludes the proof.
\end{proof}
We remark that the reformulation proposed in Theorem~\ref{thm:drinc} is exact whenever the true shadow cost of robustification $\lambda$ is greater or equal than $\lambda_{\max}(Q)$, a condition which is always satisfied for sufficiently small $\epsilon$ as per Proposition~\ref{prop:lambda_is_big_enough_if_support_is}. When $\lambda$ is lower than $\lambda_{\max}(Q)$, the solution computed using Theorem \ref{thm:drinc} may instead be suboptimal. Nevertheless, our solution retains safety and stability guarantees in face of the uncertain distribution, since neither
\eqref{eq:lem_quad_constraints} nor the achievability constraints depend on $\lambda$.



\section{Numerical results}
\label{sec:numerical}
In this section, we validate the effectiveness of the proposed control approach by benchmarking the policy $\bm{\pi}_{\operatorname{drinc}}$ synthesized using Theorem~\ref{thm:drinc} against the classical robust and LQG controllers (dubbed $\bm{\pi}_{\operatorname{rob}}$ and $\bm{\pi}_{\operatorname{LQG}}$) \cite{hassibi1999indefinite}, the distributionally robust LQG controller $\bm{\pi}_{\operatorname{DRLQG}}$ of \cite{taskesen2024distributionally}, and a certainty-equivalence policy $\bm{\pi}_{\operatorname{emp}}$, which is an instance of DRInC with $\epsilon \to 0$ in \eqref{eq:ambiguity_set_definition}. For our experiments\footnote{The source code that reproduces our numerical examples is available at \href{https://github.com/DecodEPFL/DRInC/tree/jsb_dev}{https://github.com/DecodEPFL/DRInC}.}, we consider a discrete-time double integrator system described by the linear dynamics
\begin{equation*}
    x_{t+1} = \begin{bmatrix}
        1 & 1\\ 0 & 1
    \end{bmatrix} x_t + \begin{bmatrix}
        0\\1
    \end{bmatrix} u_t + w_t\,, ~~
    y_t = \begin{bmatrix}
        1 & 0
    \end{bmatrix} x_t + v_t\,. 
\end{equation*}
We fix $T = 9$ and $D = \operatorname{blkdiag}(1, 4, 1)$ in \eqref{eq:quad_expectation_risk_def}, and we assume that the joint disturbance vector $\bm{\xi}$ takes value in $\bm{\Xi} = [-0.2, 1.0]^{d}$. We define the safe set for the state signal as $\mathcal{X} = \{(x_1, x_2) \in \mathbb{R}^2 : -6.4 \leq x_1 \leq 6.4, ~ -64 \leq x_2 \leq 64\}$, which can be defined by $4$ affine constraints. For simplicity, we do not consider input constraints, that is, we let $\mathcal{U} = \mathbb{R}$. With this definition at hand, we set the objective of satisfying \eqref{eq:cvar_constraint_definition} with $\gamma = 0.1$. Further, we assume that $N = 100$ disturbance trajectories as per \eqref{eq:training_samples} have been recorded and are available for control purposes; we use these samples to estimate $\widehat{\mathbb{P}}$ as per \eqref{eq:empirical_center_distribution} for designing $\bm{\pi}_{\operatorname{drinc}}$ and $\bm{\pi}_{\operatorname{emp}}$, and to estimate an empirical noise covariance matrix for designing $\bm{\pi}_{\operatorname{LQG}}$ and $\bm{\pi}_{\operatorname{DRLQG}}$. %

\subsection{The value of the empirical distribution}
In our first experiment, we evaluate the performance of the control policies mentioned above as the Wasserstein distance between the true probability distribution of $\xi_t = [w_t^\top, v_t^\top]^\top$ and $\widehat{\mathbb{P}}$ increases. We construct the reference distribution $\widehat{\mathbb{P}}$ and the noise covariance matrices by drawing $N = 100$ samples either from the truncated bimodal distribution shown in Figure~\ref{fig:training_samples_bimodal} or from the beta distribution shown in Figure~\ref{fig:training_samples_beta}. For designing $\bm{\pi}_{\operatorname{drinc}}$ and $\bm{\pi}_{\operatorname{DRLQG}}$, we set the radius $\epsilon$ of the ambiguity sets in Theorem~\ref{thm:drinc} and \cite[Proposition~1]{taskesen2024distributionally} equal to $0.1$, and we obtain the policy $\bm{\pi}_{\operatorname{emp}}$ as limit cases with $\epsilon \to 0$. Last, we synthesize the policy $\bm{\pi}_{\operatorname{rob}}$ by solving the following optimization problem using sampling:
\begin{subequations}
\begin{align}
    & ~ \argmin_{\bm{\Phi} \mt{ achievable}} ~
    \sup_{
    \bm{\xi}_{T} \in \bm{\Xi}
    } ~ 
    \bm \xi_{T}^\top \bm \Phi^\top D \bm \Phi \bm \xi_{T}\\
    &\st ~ \bm \Phi_x \bm \xi_{T} \in \mathcal{X}\,, \forall \bm{\xi}_T \in \bm{\Xi}\,.
\end{align}
\end{subequations}

\begin{figure}
    \centering
    \subfloat[Truncated bimodal distribution.]{\label{fig:training_samples_bimodal}
\includegraphics[width=0.24\textwidth,trim={0 0 0pt 0},clip]{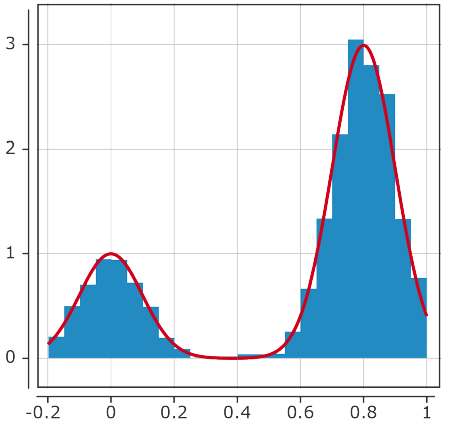}
}
     \subfloat[Beta distribution: $\alpha = \beta = 0.5$.]{\label{fig:training_samples_beta}
\includegraphics[width=0.24\textwidth,trim={0 0 0pt 0},clip]{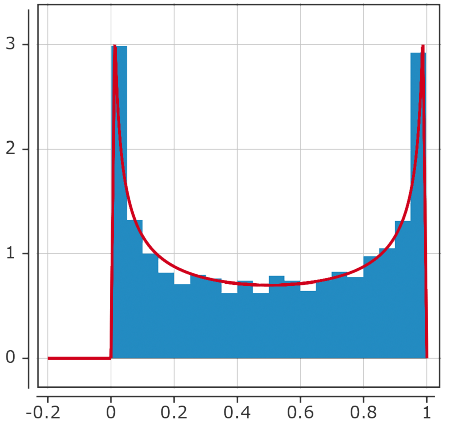}
}
    \caption{Probability density function of the real distribution $\bb P$ in red and the histogram of the empirical training distribution $\widehat{\bb P}$ for our first set of experiments in blue, marginalized to any dimension.}
    \label{fig:training_samples}
\end{figure}

\begin{figure*}
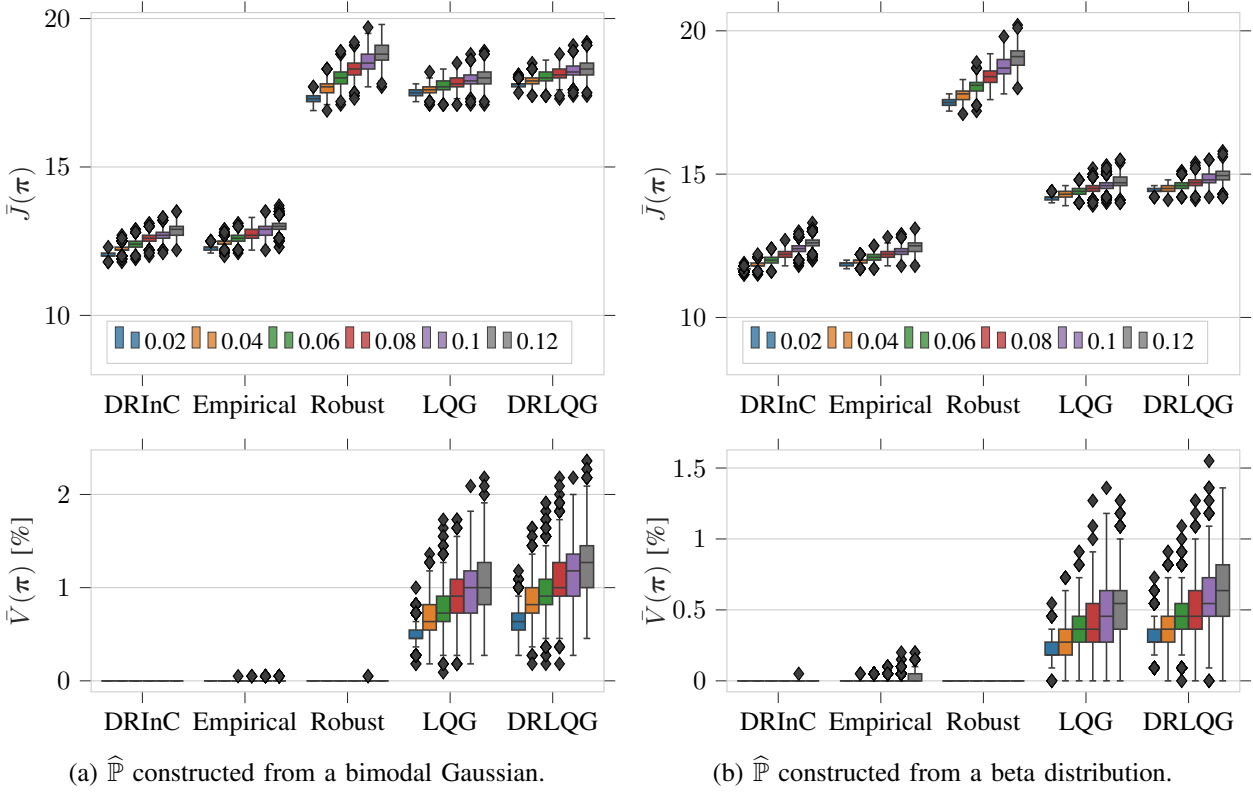

    \captionsetup[subfloat]{labelfont=normalsize,textfont=normalsize}
    \subfloat[$\widehat{\mathbb{P}}$ constructed from a bimodal Gaussian.]{
    \label{fig:performance_comparison_random_sampling_bimodal}
    \input{figs/random/double_integrator_bimodal_gaussian}}
    \subfloat[$\widehat{\mathbb{P}}$ constructed from a beta distribution.]{
    \label{fig:performance_comparison_random_sampling_beta}
    \input{figs/random/double_integrator_beta}}
    \caption{Closed-loop performance comparison: average control cost and constraint violations when $\mathbb{P}$ is given by Figure~\ref{fig:training_samples} and $\mathbb{Q}$ is randomly sampled from $\mathbb{B}_\epsilon\big(\widehat{\mathbb{P}}\big)$ as per \eqref{eq:ambiguity_set_definition}. For increasing values of $\epsilon$, we use boxes of different colors to display the difference between the first and third quartiles of the data, using a solid line to denote the median value. The whiskers extend from the box to the farthest data point lying within 1.5x the inter-quartile range from the box, whereas the diamonds represent flier points that lie past the end of the whiskers.}
    \label{fig:performance_comparison_random_sampling}

\end{figure*}

For the testing phase, we generate $1000$ randomly selected empirical distributions that lie at a fixed Wasserstein distance\footnote{To generate each testing distribution, we first split the $N$ samples in $\widehat{\mathbb{P}}$ in $\bar{N} = 10000$ to obtain more diverse distributions. We then move each sample by a random value in $[-1,1]^d$, project the results on $\bm{\Xi}$, and repeat this process until the new distribution is outside of $\bb{B}_\epsilon\big(\widehat{\mathbb{P}}\big)$. After sampling, we use a geodesic interpolation onto the Wasserstein sphere centered on $\widehat{\mathbb{P}}$ and of a particular radius \cite{geodesic_interp}.} from $\widehat{\mathbb{P}}$; in particular, we progressively increase $W^2\big(\widehat{\mathbb{P}}, \mathbb{Q}\big)$ starting from $0.02$ up to $0.12$. With these disturbance trajectories at hand, we simulate the evolution of the system \eqref{eq:system_dynamics_definition}, initialized at $x_0 = 0$, under all considered control policy, and record the average control cost $\bar{J}(\bm{\pi})$, that is, 
\begin{equation*}
    \bar{J}(\bm{\pi}) = \frac{1}{\bar{N} \bar{T}} \sum_{i = 1}^{\bar{N}} \sum_{t = 1}^{\bar{T}} \begin{bmatrix}
        x_{t, i}^\top & u_{t, i}^\top
    \end{bmatrix} D
    \begin{bmatrix}
        x_{t, i}\\ u_{t, i}
    \end{bmatrix}\,, ~~ \bar{T} = 50\,,
\end{equation*}
and the average number $\bar{V}(\bm{\pi}) \in [0, 4]$ of safety constraint violations $x_{t, i} \in \mathcal{X}$ at each time step. We collect our results in the form of box plots in Figure~\ref{fig:performance_comparison_random_sampling}. In both cases, we can observe how extracting only second-order information from the uncertainty samples, as per $\bm{\pi}_{\operatorname{LQG}}$ and $\bm{\pi}_{\operatorname{DRLQG}}$, while neglecting the presence of two modes results in a higher control cost. Similarly, failing to exploit the availability of the training samples \eqref{eq:training_samples} and adopting a worst-case approach as per $\bm{\pi}_{\operatorname{rob}}$ also results in overly conservative decision-making. At the same time, these plots do not highlight a significant performance difference between $\bm{\pi}_{\operatorname{drinc}}$ and $\bm{\pi}_{\operatorname{emp}}$, even though the lack of robustification in the latter produces slightly more constraint violations. While this may seem counterintuitive, we remark that each testing empirical distribution lies in a space of dimension $d=30$. Hence, randomly sampling in $\mathbb{B}_{\epsilon}\big(\widehat{\mathbb{P}}\big)$ is unlikely to return the most averse distribution the policy $\bm{\pi}_{\operatorname{drinc}}$ was designed to counteract.

\begin{figure*}
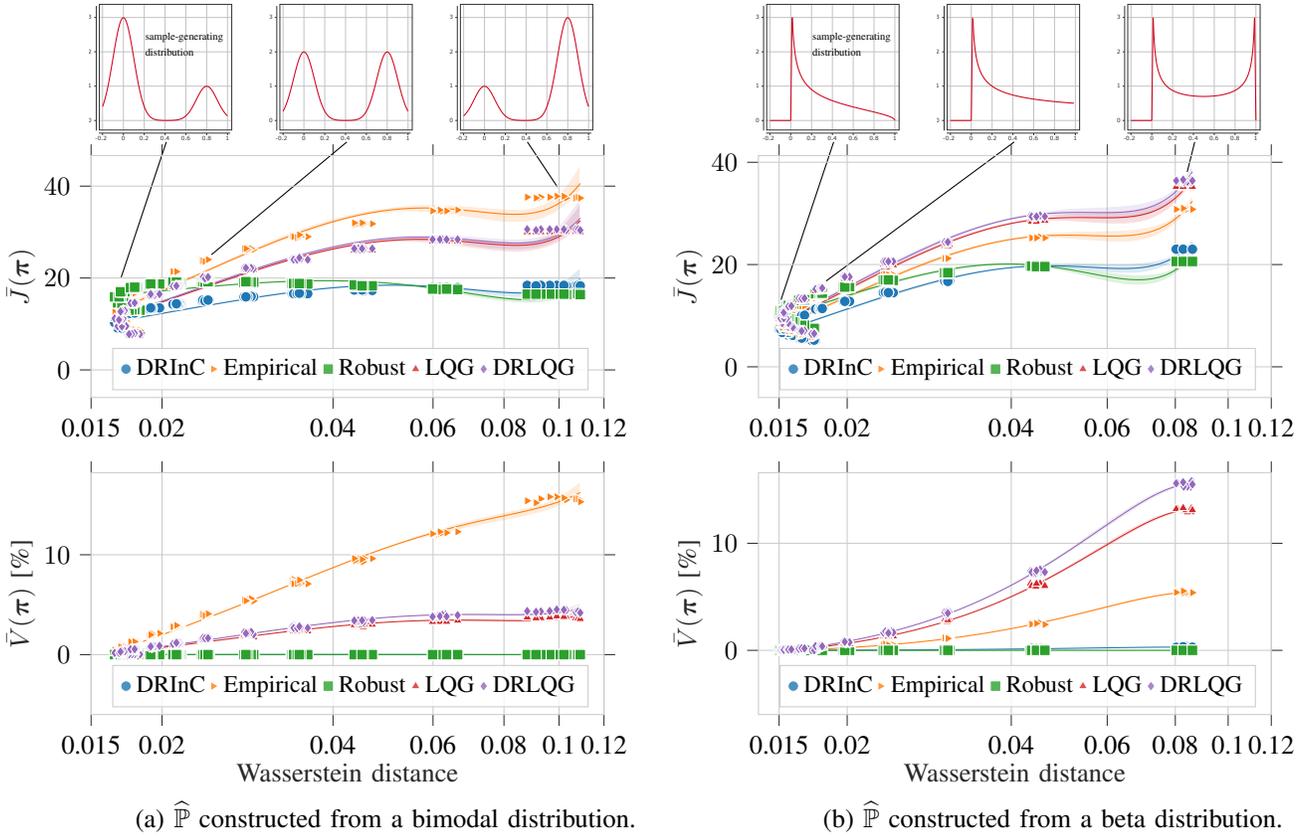

    \captionsetup[subfloat]{labelfont=normalsize,textfont=normalsize}
    \subfloat[$\widehat{\mathbb{P}}$ constructed from a bimodal distribution.]{
    \label{fig:performance_comparison_fixed_structure_bimodal}
    \input{figs/default/double_integrator_bimodal_gaussian}}
    \hspace{-50pt}
    \subfloat[$\widehat{\mathbb{P}}$ constructed from a beta distribution.]{
    \label{fig:performance_comparison_fixed_structure_beta}
    \input{figs/default/double_integrator_beta}}
    \caption{Closed-loop performance comparison: average control cost and constraint violations when $\widehat{\mathbb{P}}$ is constructed from samples from the training distributions in the top left boxes, and $\mathbb{Q}$ from samples from increasingly more different distributions in the same class (examples of such distributions are shown in the small panels on top of the main figure). We use markers of different colors to display the performance of each method, and we plot the trend as a third-order polynomial using solid lines. As both $\bm{\pi}_{\operatorname{drinc}}$ and $\bm{\pi}_{\operatorname{rob}}$ do not violate the safety constraints, the blue and green trend lines associated with these two policies on the bottom plot overlap.}
    \label{fig:performance_comparison_fixed_structure}
\end{figure*}

\subsection{The value of distributional robustness}
In our second experiment, to highlight the advantages brought about by robustifying against distributions in $\mathbb{B}_{\epsilon}\big(\widehat{\mathbb{P}}\big)$, we repeat our performance comparison when keeping the parametric structure of the probability distribution generating the disturbances fixed between training and test. In particular, for control design, we assume availability of samples from the truncated bimodal distribution and the beta distribution shown in the top left box of Figure~\ref{fig:performance_comparison_fixed_structure_bimodal} and Figure~\ref{fig:performance_comparison_fixed_structure_beta}, respectively. Then, we evaluate the performance of each control policy, in terms of both realized control cost $\bar{J}(\bm{\pi})$ and constraint violations rate $\bar{V}(\bm{\pi})$, as part of the probability mass concentrated in the mode centered at $0$ is moved from the second pile centered at $0.8$ in Figure~\ref{fig:performance_comparison_fixed_structure_bimodal}, and as the parameter $\beta$ vary from $1.5$ to $0.5$ in Figure~\ref{fig:performance_comparison_fixed_structure_beta}. We report our numerical findings in the form of scatter plot in the for main panels of Figure~\ref{fig:performance_comparison_fixed_structure}, and include a third-order polynomial trend line for interpretability. Interestingly, this second experiment clearly illustrates how the closed-loop performance of $\bm{\pi}_{\operatorname{emp}}$ rapidly degrades as the Wasserstein distance between the training and test distributions $\widehat{\mathbb{P}}$ and $\mathbb{Q}$ increases, highlighting the value of distributional robustness guarantees in mitigating the impact of any distribution shifts. Further, we can also observe that, while in terms of constraint violations $\bm{\pi}_{\operatorname{drinc}}$ and $\bm{\pi}_{\operatorname{rob}}$ perform similarly, in terms of closed-loop control cost the distributionally robust approach achieves higher performance when the Wasserstein distance between $\widehat{\mathbb{P}}$ and $\mathbb{Q}$ is sufficiently small, and that $\bm{\pi}_{\operatorname{rob}}$ starts outperforming $\bm{\pi}_{\operatorname{drinc}}$ over as we approach the boundaries of $\mathbb{B}_{\epsilon}\big(\widehat{\mathbb{P}}\big)$.

\section{Conclusion}
We have presented an end-to-end synthesis method from a collection of a finite number of disturbance realizations to the design of a stabilizing linear policy with DR safety and performance guarantees. Our approach consists in estimating an empirical distribution using samples of the uncertainty, and then computing a feedback policy that safely minimizes the worst-case expected cost over all distributions within a Wasserstein ball around the nominal estimate through the solution of an SDP. We have shown that, as the radius of this ambiguity set varies, our problem statement recovers classical control formulations. To address the resulting optimal control problem, we have established a novel tight convex relaxation for DRO of quadratic objectives. Then, we have combined our results with the system-level synthesis framework, presenting conditions under which our design method is non-conservative. Finally, we have demonstrated the value of exploiting past disturbance trajectories whenever available and of distributional robustness through numerical simulations. Interesting avenues for future research include methods for adapting the reference distribution $\widehat{\mathbb{P}}$ and the radius $\epsilon$ online based on newly observed disturbances samples, and extending the design to system with nonlinear dynamics.

\bibliographystyle{IEEEtran}
\bibliography{references}

\appendix\section{Appendix}
\subsection{SLS controller implementation}
\label{app:controller_implementation}
The following proposition shows how to implement a controller that achieves a given pair of system responses $\bm{\Phi}_x$ and $\bm{\Phi}_u$.
\begin{proposition} \label{prop:controller_implementation}
    If the closed loop map $\bm \Phi$ is achievable, the corresponding control policy $\bm \pi(\bm \Phi)$ can be implemented as a linear system with dynamics
\vspace{10pt}
\begin{align}\label{eq:controller_implementation}
    \bm A_K &\!=\! \lt[\matb \,\smash[t]{\smalloverbrace{0}^{n}}\!\! & \smash[t]{\smalloverbrace{I}^{\!\!n(T-1)\!\!}}\!\! & 
    \smash[t]{\smalloverbrace{0}^{p}} & 
    \smash[t]{\smalloverbrace{0}^{pT}}\!\!
    \\ 0\!\! & -\bm\Phi_1 & 0\!\! & -\bm\Phi_3
    \\
    0\!\! & 0 & 0\!\! & I \\ 0\!\! & 0 & 0\!\! & 0\mate\rt]
    \!\!,\,
    \bm B_K \!=\! \lt[\matb 0 \\ -\bm\Phi_4 \\ 0 \\ I \mate\rt] \begin{matrix*}[l] \}\,  n(T-1) \\ \}\, n\quad \\ \}\, pT \quad \\ \}\, p \quad \end{matrix*} \!,
    \nonumber \\
    \bm C_K &\!= \big[
    \bm\Phi_5, \bm\Phi_8 C, \bm\Phi_7, \bm\Phi_8 
    \big],
\end{align}
where $
    \bm\Phi = \lt[\matb \bm\Phi_1& \bm\Phi_2& \bm\Phi_3& \bm\Phi_4 
    \\ 
    \smash[b]{\underbrace{\bm\Phi_5}_{n(T-1)}}& \smash[b]{\underbrace{\bm\Phi_6}_{2n}}& \smash[b]{\underbrace{\bm\Phi_7}_{pT}}& \smash[b]{\underbrace{\bm\Phi_8}_{p} }\mate\rt].
$
\vspace{24pt}
\end{proposition}

\begin{proof}
From \cite{wang2019system},
\begin{align*}
    \bm\delta &= (I - z\bm\Phi_{xw}(z)) \bm\delta - \bm\Phi_{xv}(z) \bm y,
    \\
    \bm u &= z \bm \Phi_{uw}(z)\bm \delta + \bm\Phi_{uv}(z) \bm y,
\end{align*}
which means that at each timestep $t \geq T$, one has
\begin{align}
    \delta_{t} &= \delta_{t} - \sum_{k=1}^T \Phi_{xw}(k) \delta_{t-k+1} - \sum_{k=1}^T \Phi_{xv}(k) y_{t-k},
    \nonumber \\
    u_{t} &= \sum_{k=1}^T \Phi_{uw}(k) \delta_{t-k+1} + \sum_{k=0}^T \Phi_{uv}(k) y_{t-k}.
    \label{eq:appendix_controller_implementation_filter}
\end{align}
The achievability constraints \eqref{eq:achievability_constraints_matrix} imply $\Phi_{xw}(1) = I$ and $\Phi_{uw}(1) = \Phi_{uv}(0)C$, (see Appendix \ref{app:details_achievability_constraints_matrix}). Hence, \eqref{eq:appendix_controller_implementation_filter} can be reformulated at as
\begin{align*}
    \delta_{t} &= -\sum_{k=1}^{T-1} \Phi_{xw}(k+1) \delta_{t-k} - \sum_{k=1}^T \Phi_{xv}(k) y_{t-k},
    \\
    u_{t} &= \Phi_{uv}(0)C \delta_t + \sum_{k=1}^{T-1} \Phi_{uw}(k+1) \delta_{t-k} + \sum_{k=0}^{T} \Phi_{uv}(k) y_{t-k}.
\end{align*}
Writing this controller implementation in matrix form and noting that $\Phi_{xv}(0) = 0$ yields 
\begin{align}
\label{eq:controller_implementation_almost_there}
    \delta_{t} = -\bm \Phi_x \bm \phi_{t-T:t} ,\,
    u_{t} = \bm \Phi_u \bm \phi_{t-T:t} + \Phi_{uv}(0)C \delta_t,
\end{align}
where $\bm \phi_{t-T:t} = [\delta_{t-T+1}^\top, \dots, \delta_{t-1}^\top, 0_{2n}, y_{t-T}^\top, \dots, y_{t}^\top]^{\!\top}$. Contructing a system with states $\bm \phi_{t-T:t}$, input $y_t$, output $u_t$, and respecting \eqref{eq:controller_implementation_almost_there}  yields the dynamics \eqref{eq:controller_implementation}, which concludes the proof.
\end{proof}

\subsection{Infinite horizon Achievability} \label{app:details_achievability_constraints_matrix}
\begin{proposition}
    The achievability constraints \eqref{eq:achievability_constraints} are equivalent to
    \begin{subequations}
    \label{eq:achievability_constraints_matrix}
    \begin{align}
        [I_n,0]\bm\Phi \!\lt[\matb \mc Z^- \!\otimes\! I_n & \!\!\!\!0 \\ 0 & \!\!\!\!\!\!\mc Z^- \!\otimes\! I_p \mate\rt] &\!=\, 
        [A,B]\bm\Phi \!\lt[\matb \mc Z^+ \!\otimes\! I_n & \!\!\!\!0 \\ 0 & \!\!\!\!\!\!\mc Z^+ \!\otimes\! I_p \mate\rt] 
        \nonumber \\
        &\!\!\!\!\!\!\!\!+ [\mc Z^+_{T+1} \!\otimes\! I_n, \mc Z^+_{T+1} \!\otimes\! 0C^{\!\top}],\!\!
        \\
        \bm\Phi\!\lt[\matb \mc Z^- \!\otimes\! I_n \\ \mc Z^- \!\otimes\! (0C) \mate\rt] \!=\,& 
        \bm\Phi\!\lt[\matb \mc Z^+ \!\otimes\! A \\ \mc Z^+ \!\otimes\! C \mate\rt]
        \nonumber \\
        &\!\!\!\!+ \lt[\matb\mc Z^+_{T+1} \!\otimes\! I_n \\ \mc Z^+_{T+1} \!\otimes\! (0B^\top) \mate\rt]\!,
    \end{align}
    \end{subequations}
    where $\mc Z^+ = [I_{T+1}, 0]$, $\mc Z^- = [0, I_{T+1}]$ are in $\bb R^{(T+1)\times(T+2)}$, and $\mc Z^+_{T+1}$ is the last row of $\mc Z^+$.
\end{proposition}
\begin{proof}
By treating $\bm{\Phi}_x$ and $\bm{\Phi}_u$ as FIR filters:
\begin{align*}
    &\sum_{k = 1}^T \! \Phi_{xw}(k) z^{-k+1} \!-\! A \Phi_{xw}(k) z^{-k} \!-\! B \Phi_{uw}(k) z^{-k} \!=\! I,\\
    &\sum_{k = 1}^T \! \Phi_{xv}(k) z^{-k+1} \!-\! A \Phi_{xv}(k) z^{-k} \!-\! B \Phi_{uv}(k) z^{-k} \!=\! B \Phi_{uv}(0),\\
    &\sum_{k = 1}^T \! \Phi_{xw}(k) z^{-k+1} \!-\! \Phi_{xw}(k)A z^{-k} \!-\! \Phi_{xv}(k) C z^{-k} \!=\! I,\\
    &\sum_{k = 1}^T \! \Phi_{uw}(k) z^{-k+1} \!-\! \Phi_{uw}(k)A z^{-k} \!-\! \Phi_{uv}(k) C z^{-k} \!=\! \Phi_{uv}(0) C,
\end{align*}
which is equivalent to
\begin{subequations}
\begin{align}
    \label{eq:achievability_constraint_proof_a}
    &\Phi_{xw}(0) \!=\! 0, \Phi_{xv}(0) \!=\! 0, \Phi_{uw}(0) \!=\! 0,
    \\ \label{eq:achievability_constraint_proof_b}
    &\Phi_{xw}(1) \!=\! I, \Phi_{xv}(1) \!=\! B \Phi_{uv}(0), \Phi_{uw}(1) \!=\! \Phi_{uv}(0) C,
    \\ \label{eq:achievability_constraint_proof_c}
    &\Phi_{xw}(k+1) \!=\! A \Phi_{xw}(k) \!+\! B \Phi_{uw}(k)\,, \forall k = 1,\dots, T,
    \\ \label{eq:achievability_constraint_proof_d}
    &\Phi_{xv}(k+1) \!=\! A \Phi_{xv}(k) \!+\! B \Phi_{uv}(k)\,, \forall k = 1,\dots, T,
    \\ \label{eq:achievability_constraint_proof_e}
    &\Phi_{xw}(k+1) \!=\! \Phi_{xw}(k) A \!+\! \Phi_{xv}(k) C\,, \forall k = 1,\dots, T,
    \\ \label{eq:achievability_constraint_proof_f}
    &\Phi_{uw}(k+1) \!=\! \Phi_{uw}(k) A \!+\! \Phi_{uv}(k) C\,, \forall k = 1,\dots, T,
    \\ \label{eq:achievability_constraint_proof_g}
    &\Phi(T+1) \!=\! 0.
\end{align}
\end{subequations}
In matrix form, this yields
\begingroup 
\setlength\arraycolsep{4pt}
\begin{align*}
&[I, 0]\bm\Phi
\matb
\lt[\matb\!
    \lt[\matb
    0 & I & \dots & 0 & 0 & 0
    \\[-4pt]
    \vdots & \vdots & \ddots & \vdots & \vdots & \vdots
    \\
    0 & 0 & \dots & I & 0 & 0
    \\
    0 & 0 & \dots & 0 & I & 0
    \\
    0 & 0 & \dots & 0 & 0 & I
    \mate\rt] & \!\!\!\!\!0
    \\
    0 & \!\!\!\!\!\lt[\matb
    0 & I & \dots & 0 & 0 & 0
    \\[-4pt]
    \vdots & \vdots & \ddots & \vdots & \vdots & \vdots
    \\
    0 & 0 & \dots & I & 0 & 0
    \\
    0 & 0 & \dots & 0 & I & 0
    \\
    0 & 0 & \dots & 0 & 0 & I
    \mate\rt]\!\!
\mate\rt]
\\
\aunderbrace[l1r]{}[D]_{\eqref{eq:achievability_constraint_proof_g}}\;
\underbrace{\hspace{40pt}}_{\eqref{eq:achievability_constraint_proof_c}}\;
\aunderbrace[l1r]{}[D]_{\eqref{eq:achievability_constraint_proof_b}}\;
\aunderbrace[l1r]{}[D]_{\eqref{eq:achievability_constraint_proof_a}}
\quad
\aunderbrace[l1r]{}[D]_{\eqref{eq:achievability_constraint_proof_g}}\;
\underbrace{\hspace{40pt}}_{\eqref{eq:achievability_constraint_proof_d}}\;
\aunderbrace[l1r]{}[D]_{\eqref{eq:achievability_constraint_proof_b}}\;
\aunderbrace[l1r]{}[D]_{\eqref{eq:achievability_constraint_proof_a}}
\mate
\\
& = \big[[0, 0, \dots, 0, I, 0], [0, 0, \dots, 0, 0, 0]\big] +
\\
&[A,B]\bm\Phi
\matb
\lt[\matb\!
    \lt[\matb
    I & 0 & \dots & 0 & 0 & 0
    \\
    0 & I & \dots & 0 & 0 & 0
    \\[-4pt]
    \vdots & \vdots & \ddots & \vdots & \vdots & \vdots
    \\
    0 & 0 & \dots & I & 0 & 0
    \\
    0 & 0 & \dots & 0 & I & 0
    \mate\rt] & \!\!\!\!\!0
    \\
    0 & \!\!\!\!\!\lt[\matb
    I & 0 & \dots & 0 & 0 & 0
    \\
    0 & I & \dots & 0 & 0 & 0
    \\[-4pt]
    \vdots & \vdots & \ddots & \vdots & \vdots & \vdots
    \\
    0 & 0 & \dots & I & 0 & 0
    \\
    0 & 0 & \dots & 0 & I & 0
    \mate\rt]\!\!
\mate\rt]
\\
\aunderbrace[l1r]{}[D]_{\eqref{eq:achievability_constraint_proof_g}}\;
\underbrace{\hspace{40pt}}_{\eqref{eq:achievability_constraint_proof_c}}\;
\aunderbrace[l1r]{}[D]_{\eqref{eq:achievability_constraint_proof_b}}\;
\aunderbrace[l1r]{}[D]_{\eqref{eq:achievability_constraint_proof_a}}
\quad
\aunderbrace[l1r]{}[D]_{\eqref{eq:achievability_constraint_proof_g}}\;
\underbrace{\hspace{40pt}}_{\eqref{eq:achievability_constraint_proof_d}}\;
\aunderbrace[l1r]{}[D]_{\eqref{eq:achievability_constraint_proof_b}}\;
\aunderbrace[l1r]{}[D]_{\eqref{eq:achievability_constraint_proof_a}}
\mate,
\end{align*}
\endgroup
and
\begingroup 
\setlength\arraycolsep{4pt}
\begin{align*}
\bm\Phi
\matb
\lt[\matb\!
    \lt[\matb
    0 & I & \dots & 0 & 0 & 0
    \\[-4pt]
    \vdots & \vdots & \ddots & \vdots & \vdots & \vdots
    \\
    0 & 0 & \dots & I & 0 & 0
    \\
    0 & 0 & \dots & 0 & I & 0
    \\
    0 & 0 & \dots & 0 & 0 & I
    \mate\rt]\!\!
    \\
    \!\lt[\matb
    0 & 0 & \dots & 0 & 0 & 0
    \\[-4pt]
    \vdots & \vdots & \ddots & \vdots & \vdots & \vdots
    \\
    0 & 0 & \dots & 0 & 0 & 0
    \\
    0 & 0 & \dots & 0 & 0 & 0
    \\
    0 & 0 & \dots & 0 & 0 & 0
    \mate\rt]\!\!
\mate\rt]
\\
\,
\aunderbrace[l1r]{}[D]_{\eqref{eq:achievability_constraint_proof_g}}\;
\underbrace{\hspace{40pt}}_{\eqref{eq:achievability_constraint_proof_e}, \eqref{eq:achievability_constraint_proof_f}}\;
\aunderbrace[l1r]{}[D]_{\eqref{eq:achievability_constraint_proof_b}}\;
\aunderbrace[l1r]{}[D]_{\eqref{eq:achievability_constraint_proof_a}}
\quad
\mate
=
\setlength\arraycolsep{3pt}
\bm\Phi
\matb
\lt[\matb\!
    \lt[\matb
    A & 0 & \dots & 0 & 0 & 0
    \\
    0 & A & \dots & 0 & 0 & 0
    \\[-4pt]
    \vdots & \vdots & \ddots & \vdots & \vdots & \vdots
    \\
    0 & 0 & \dots & A & 0 & 0
    \\
    0 & 0 & \dots & 0 & A & 0
    \mate\rt]\!\!
    \\
    \!\lt[\matb
    C & 0 & \dots & 0 & 0 & 0
    \\
    0 & C & \dots & 0 & 0 & 0
    \\[-4pt]
    \vdots & \vdots & \ddots & \vdots & \vdots & \vdots
    \\
    0 & 0 & \dots & C & 0 & 0
    \\
    0 & 0 & \dots & 0 & C & 0
    \mate\rt]\!\!
\mate\rt]
\\
\,
\aunderbrace[l1r]{}[D]_{\eqref{eq:achievability_constraint_proof_g}}\;
\underbrace{\hspace{40pt}}_{\eqref{eq:achievability_constraint_proof_e}, \eqref{eq:achievability_constraint_proof_f}}\;
\aunderbrace[l1r]{}[D]_{\eqref{eq:achievability_constraint_proof_b}}\;
\aunderbrace[l1r]{}[D]_{\eqref{eq:achievability_constraint_proof_a}}
\quad
\mate
\\
+ \lt[\matb 0, 0, \dots, 0, I, 0 \\ 0, 0, \dots, 0, 0, 0 \mate\rt] \!.
\end{align*}
\endgroup
The matrices can be written in a compact form as \eqref{eq:achievability_constraints_matrix}, which concludes the proof.
\end{proof}

\subsection{Proof of Proposition \ref{prop:cvar_original_deterministic_constraint_bounded_support}}
\label{app:proof_prop_bounded_support}


The risk \eqref{eq:prop_dro_general_risk_def} is contingent on three mathematical objects:
\begin{enumerate}[label=(\roman*)]
    \item A loss function $\max\limits_{j\in[J]} \ell_j(\bm\xi_T) = \max\limits_{j\in[J]} a_j^\top \bm\xi_T + b_j$,
    \item A transport cost $c(\bm\xi_T,\bm\xi_T^{(i)}) = \|\bm\xi_T - \bm\xi_T^{(i)}\|_2^2$,
    \item and a support $\bm\Xi = \Big\{\bm \xi : \max\limits_{k\in [n_H]} f_k( \bm \xi) \leq 0\Big\}$, where $n_H$ is the number of rows in $H$ and $f_k(\bm \xi) = H_k \bm \xi - h_k$.
\end{enumerate}
Moreover, since the loss is concave and both the transport cost and the support are convex, \eqref{eq:prop_dro_general_risk_def} shows strong duality properties if and only if it is strictly feasible. The strict feasibility is guaranteed by the full-dimensionality of $\bm\Xi$ and the strict positivity of $\epsilon$. The dual problem is given by \cite{shafieezadeh2023new} as
\begin{align*}
    &\inf_{s^{(i)}, \lambda \geq 0} \lambda \epsilon + \frac{1}{N}\sum_{i=0}^N s^{(i)},
    \\
    &\st \sup_{\bm \xi_T \in \bm \Xi} \ell(\bm \xi_T) - \lambda c(\bm \xi_T,\bm \xi_T^{(i)}) \leq s^{(i)} \,,\forall i \in [N].
\end{align*}
While the dual problem does not seem much simpler to solve than the primal at first glance, we use \cite[Proposition~2.12]{shafieezadeh2023new} to reformulate it using convex conjugates. In our own notation, this gives
%
\begin{align}
    &\!\inf_{s^{(i)}, \lambda \geq 0 \,, \kappa_{ijk} \geq 0}\! \lambda \epsilon \!+ \!\frac{1}{N} \!\!\sum_{i \in [N]}\!\! s^{(i)} \!,\, \st \,, \forall i \!\in\! [N], \forall j \!\in\! [J]\!:
    \nonumber \\ \nonumber 
    &s^{(i)} \geq (-\ell_{j})^\star(\zeta_{ij}^\ell)
    + \lambda c^{\star} \!\lt(\! \frac{\zeta_{ij}^c}{\lambda}, \widehat{\bm \xi}_T^{(i)} \!\!\rt)\! + \!\!\sum_{k \in [n_H]}\!\! \kappa_{ijk} f_k^\star \!\lt(\! \frac{\zeta_{ijk}^f}{\kappa_{ijk}} \!\rt) \!\!,
    \\ \label{eq:prop_212}  
    & \zeta_{ij}^\ell + \zeta_{ij}^c + \sum_{k \in [n_H]} \zeta_{ijk}^f = 0,
\end{align}
where $(-\ell_{j})^\star$ is the convex conjugate of the opposite of $\ell_j$, $c^\star$ is the convex conjugate of the transport cost $c$ with respect to the first argument, and $f_k^{\star}$ is the convex conjugate of $f_k$. Note that the case where $\lambda = 0$ is also well defined in \cite{shafieezadeh2023new} despite the division. All three functions are either linear or quadratic so their conjugates are well-known \cite{borwein2005convex}. Both $-\ell_j$ and $f_k$ are linear so their convex conjugates are $b_j$ and $h_k$ if the conjugates' arguments are equal to $-a_j$ and $H_k$, respectively, and infinite otherwise. The conjugate of the transport cost is given by $c^\star(\zeta, \bm \xi_T^{(i)}) = \frac{1}{4} \zeta^\top \zeta + \zeta^\top \bm \xi_T^{(i)}$. 

In order to minimize \eqref{eq:prop_212}, one must avoid infinite costs, which adds constraints on $\lambda$, $\zeta_{ij}^\ell$, and $\zeta_{ijk}^f$. This means that \eqref{eq:prop_212} is equivalent to
\begin{subequations}\label{eq:proof_prop_212}
\begin{align}\label{eq:proof_prop_212_cost}
    &\inf_{s^{(i)}, \lambda \geq 0 \,, \kappa_{ijk} \geq 0} \lambda \epsilon + \frac{1}{N}\sum_{i \in [N]} s^{(i)} \leq 0 \,,
    \\ \label{eq:proof_prop_212_main_inequality}
    &\st \,, \forall i \in [N] \,, \forall j \in [J] \!:
    \nonumber \\
    &s^{(i)} \geq b_j
    + \frac{1}{4\lambda} (\zeta_{ij}^c)^\top \zeta_{ij}^c + \zeta_{ij}^c \bm \xi_T^{(i)} + \!\!\sum_{k \in [n_H]}\!\! \kappa_{ijk} h_k ,
    \\ \label{eq:proof_prop_212_zeta_equality}
    & \zeta_{ij}^\ell + \zeta_{ij}^c + \!\!\sum_{k \in [n_H]} \!\!\zeta_{ijk}^f \!= 0 \,,
    \zeta_{ij}^\ell \!=\! -a_j \,, \zeta_{ijk}^f \!=\! \kappa_{ijk} H_k.
\end{align}
\end{subequations}
To conclude the proof, we stack $\kappa_{ijk}$ for all $k \in [n_H]$ into a vector $\kappa_{ij}$ and plug the equality constraints \eqref{eq:proof_prop_212_zeta_equality} into \eqref{eq:proof_prop_212_main_inequality} to obtain \eqref{eq:prop_dro_general_risk_dual}.

\subsection{Proof of Proposition \ref{prop:lambda_is_big_enough_if_support_is}}\label{app:lambda_is_big_enough_if_support_is}
We first note that when $\bm \Xi$ is bounded, as the probability mass cannot be moved infinitely far away, the supremum of \eqref{eq:prop_dro_general_quad_risk_def} is attained. In turn, this means that the set $\argmax_{\mathbb{Q} \in \mathbb{B}_{\epsilon}} \bb E_{\xi_T \sim \bb Q} ~  \bm\xi_T^\top Q \bm\xi_T$ is non-empty. Moreover, among the set of worst-case distributions, there always exists one which is empirical \cite{gao2023distributionally}; we denote by $\bb Q^\star$ any such empirical worst-case distribution.
Second, the lower-bounded average squared distance between the samples and the border of $\bm \Xi$ implies that no distribution in $\bb B_\epsilon\big(\widehat{\mathbb{P}}\big)$ has mass only at the border of $\bm \Xi$, as the transport cost would be greater than $\epsilon$.
 %
This means that there exists a set $\bm\Xi_\delta \subset \bm\Xi$ containing at least an amount of mass $\delta > 0$ in all $\bb Q^\star$, and which is more than $\sqrt{\delta}$ away from the boundary of $\bm\Xi$. Third and finally, let $\mathcal{W}_{\max}$ denote the eigenspace spanned by the eigenvectors of $Q$ associated with $\lambda_{\max}(Q)$.

With this notation in place, we observe that the shadow cost of robustification defined in \eqref{eq:shadow_cost_def} satisfies:
\begin{align*}
    \frac{d\mc R_\epsilon(Q)}{d\epsilon^+} &\!= \lim_{\Delta\epsilon \rightarrow 0^+} {\Delta\epsilon}^{\!-1}(\mc R_{\epsilon + \Delta\epsilon}(Q) - \mc R_\epsilon(Q)),
    \\
    &\!\geq \!\! \lim_{\Delta\epsilon \rightarrow 0^+} \!{\Delta\epsilon}^{\!-1} \!\!\!\! \sup_{\bb Q^\prime \in \bb B_{\Delta\epsilon}(\bb Q^\star)} \!\! \bb E_{\substack{\bm\xi_T^\prime \sim \bb Q^\prime \\ \bm\xi_T \sim \bb Q^\star}} {\bm\xi_T^\prime}^{\!\!\top} Q \bm \xi_T^\prime \!-\! {\bm\xi_T}^{\!\!\top} Q \bm \xi_T ,
    \\
    &\!\geq\!\! \lim_{\Delta\epsilon \rightarrow 0^+} \!{\Delta\epsilon}^{\!-1} \!\!\!\! \sup_{\bb Q^\prime \in \bb B_{\Delta\epsilon}(\bb Q^\star)} \!\! \bb E_{\substack{\bm\xi_T^\prime \sim \bb Q^\prime \\ \bm\xi_T \sim \bb Q^\star}} \! (\bm\xi_T^\prime \!-\! \bm\xi_T \!)^{\!\!\top} \! Q (\bm\xi_T^\prime \!-\! \bm\xi_T \!) 
    \\
    & \hspace{131pt} +\! 2(\bm\xi_T^\prime \!-\! \bm\xi_T \!)^{\!\!\top} Q \bm \xi_T.
\end{align*}
where the first inequality follows as $\bb B_{\Delta\epsilon}(\bb Q^\star) \subseteq \bb B_{\epsilon + \Delta\epsilon}\big(\widehat{\mathbb{P}}\big)$. 

Moreover, constraining the probability distributions to belong to a subset of $\bb B_{\Delta\epsilon}(\bb Q^\star)$ where the displacements $\Delta\bm \xi = \bm\xi_T^\prime - \bm\xi_T$ of mass $\delta$ encoded by the optimal transport map relative to $\bb Q^\star$ are uniform and parallel to $\mathcal{W}_{\operatorname{max}}$ yields the lower bound:
\begin{align*}
    \frac{d\mc R_\epsilon(Q)}{d\epsilon^+} 
    &\!\geq \delta \!\lim_{\Delta\epsilon \rightarrow 0^+} \!{\Delta\epsilon}^{\!-1} \!\!\!\! \max_{\delta \|\Delta \bm\xi\|_2^2 \leq \Delta \epsilon} \!\! \bb E_{\bm\xi_T \sim \bb Q^\star} \lambda_{\max}(Q) \|\Delta \bm\xi\|_2^2
    \\
    & \hspace{158pt} + 2 \bm\xi_T^\top Q \Delta\bm\xi ,
    \\
    &\!\geq\! \delta \!\! \lim_{\Delta\epsilon \rightarrow 0^+} \!\! {\Delta\epsilon}^{\!-1} \!\!\!\! \max_{\delta \|\Delta \bm\xi\|_2^2 \leq \Delta \epsilon} \!\! \lambda_{\max}(Q) \|\! \Delta\bm\xi\|_2^2,
    %
    %
\end{align*}
where the second inequality follows since the sign of $\Delta\bm\xi$ can always be flipped to make the expected value of $2 \bm\xi_T^\top Q \Delta\bm\xi$ positive.
Note now that only the squared norm $\|\Delta\bm \xi\|_2^2$ is constrained because the support $\bm\Xi$ is not constraining for $\Delta\epsilon \leq \delta^2$ due to the definition of $\delta$. Hence, the maximum is lower bounded by $\delta^{-1} \lambda_{\max}(Q) \Delta\epsilon$. Since the term $\Delta\epsilon^{-1}$ cancels out, it then follows that 
\begin{align*}
    \lambda^\star = \frac{d\mc R_\epsilon(Q)}{d\epsilon^+} \geq \lim_{\Delta\epsilon \rightarrow 0^+} \!{\Delta\epsilon}^{\!-1} \lambda_{\max}(Q) \Delta\epsilon = \lambda_{\max}(Q),
\end{align*}
which concludes the proof.

\subsection{Proof of Lemma \ref{coro:quad_dro_prob_is_limit_linear}}
\label{app:proof_lem_bounded_support_quad}

In order to prove Lemma \ref{coro:quad_dro_prob_is_limit_linear}, we first need the following proposition.
\begin{proposition}\label{prop:quad_dro_prob_is_limit_linear}
Under the assumptions of Lemma \ref{coro:quad_dro_prob_is_limit_linear}, the risk $\mc R_\epsilon(Q)$ defined in \eqref{eq:prop_dro_general_quad_risk_def}
satisfies
\begin{subequations}\label{eq:prop_dro_general_quad_risk_dual}
    \begin{align}
    \label{eq:prop_dro_general_quad_risk_dual_cost}
    \mc R_\epsilon(Q) \leq &\inf_{s^{(i)}, 
    \lambda \geq 0, \kappa_{i} \geq 0}  
    \lambda \epsilon + \frac{1}{N} \!\sum_{i \in [N]} s^{(i)}\,,\;
    \\ \nonumber
    &\st\,, \forall i \in [N]:
    \\ \label{eq:prop_dro_general_quad_risk_dual_inequality}
    & s^{(i)} \!\geq  \max_{\bar{\bm\xi} \in \Xi} 
    - \bar{\bm\xi}^\top (Q \!-\! \lambda^{\!-1}Q^2) \bar{\bm\xi} \!+\! 2 \bar{\bm\xi}^\top Q \bm \xi_T^{(i)}
    \\ \nonumber
    & +\! \frac{1}{4\lambda} \kappa_i^\top \! H H^\top \!\kappa_i \!-\! \frac{1}{\lambda} \bar{\bm\xi}^\top Q H^\top \!\kappa_i \!+\! \big(h \!-\! H \bm \xi_T^{(i)} \big)^{\!\!\top} \!\kappa_i \,,
    \end{align}
\end{subequations}
    Moreover, \eqref{eq:coro_dro_general_quad_risk_dual_cost} holds with equality if the optimum $\lambda^\star$ of $\lambda$ satisfies $\lambda^\star I \succeq Q$.
\end{proposition}
\begin{proof}
The proof starts by linking the formulation \eqref{eq:prop_dro_general_risk_def} for piece-wise affine costs to $\mc R_\epsilon(Q)$. To do so, we express the quadratic control cost as the envelope of its tangents at each point of a $d$-dimensional of an $\epsilon$-net, i.e., an equally-spaced grid $\mc G_J \subseteq \bm\Xi$, composed of $J$ points. Mathematically, this gives
\begin{align*}
    \bm\xi_T^\top Q \bm \xi_T = \lim_{J\rightarrow \infty} \max_{j \in [J]} \,  2\bm\xi_T^\top Q \bm\xi_j - \bm\xi_j^\top Q \bm\xi_j \,,
\end{align*}
where $\bm\xi_j$ is the $j^{th}$ element of $\mc G_J$. In order to obtain a formulation that fits \eqref{eq:prop_dro_general_risk_def}, one must show that the limit operator commutes with the supremum and the expectation. 

We show the commutation of the limit with the dominated convergence theorem \cite{bartle2014elements} by finding bounds on the piece-wise affine approximation error 
\begin{align*}
     \Delta_J &= \bm\xi_T^\top Q \bm \xi_T - \max_{j \in [J]} \,  2\bm\xi_T^\top Q \bm\xi_j - \bm\xi_j^\top Q \bm\xi_j,
     \\
     &= \bm\xi_T^\top Q \bm \xi_T + \min_{j \in [J]} \,  \bm\xi_j^\top Q \bm\xi_j - 2\bm\xi_T^\top Q \bm\xi_j,
     \\
     &= \min_{j \in [J]} \, (\bm\xi_T - \bm\xi_j)^\top Q (\bm\xi_T - \bm\xi_j) \,.
\end{align*}
Note that $\Delta_J \geq 0$ because the tangents of a quadratic function are always below the curve. Moreover, the inequality $\Delta_J \leq \lambda_{\max}(Q) \min_{j \in [J]} \|\bm\xi_T - \bm\xi_j\|_2^2$ is satisfied by definition.
 
Furthermore, the distance $\min_{j \in [J]} \|\bm\xi_T - \bm\xi_j\|_2^2$ between any $\bm\xi_T \in \bm\Xi$ and the closest point of the grid $\mc G_J \subseteq \bm\Xi$ can be bounded as
\begin{align*}
    \min_{j \in [J]} \|\bm\xi_T - \bm\xi_j\|_2^2 \leq 2r(\bm\Xi) \sqrt{d} J^{-\frac{1}{d}} \,, \forall \bm\xi_T \in \bm\Xi \,,
\end{align*}
where $r(\bm\Xi) < \infty$ is the radius of a ball containing $\bm\Xi$, which is finite because $\bm\Xi$ is bounded. This gives the following inequality
\begin{align*}
    \bm\xi_T^\top Q \bm \xi_T \!-\! \Delta_Q J^{-\frac{1}{d}} 
    \!\leq \max_{j \in [J]}   2\bm\xi_T^\top Q \bm\xi_j \!-\! \bm\xi_j^\top Q \bm\xi_j 
    \leq \bm\xi_T^\top Q \bm \xi_T \,,
\end{align*}
where $\Delta_Q = 2r(\bm\Xi) \sqrt{d} \lambda_{\max}(Q)$. The expectation preserves order so for any $\bb Q$,
\begin{align*}
    \bb E_{\xi_T \!\sim \bb Q} \bm\xi_T^\top Q \bm \xi_T \!-\! \Delta_Q J^{-\frac{1}{d}} 
    \!&\leq \bb E_{\xi_T \!\sim \bb Q} \max_{j \in [J]}   2\bm\xi_T^\top Q \bm\xi_j \!-\! \bm\xi_j^\top Q \bm\xi_j \,,
    \\
    &\quad\quad \leq \bb E_{\xi_T \!\sim \bb Q}\bm\xi_T^\top Q \bm \xi_T \,.
\end{align*}
Finally, if all points of a function satisfy an inequality, its supremum must satisfy it as well, hence
\begin{align*}
    \!\mc R_\epsilon \!(Q) \!-\! \Delta_Q J^{-\frac{1}{d}}\! 
    \!\leq\! \!\sup_{\mathbb{Q} \in \mathbb{B}_{\epsilon}} \!\! \bb E_{\xi_T \!\sim \bb Q} \!\max_{j \in [J]} 2\bm\xi_T^\top Q \bm\xi_j \!-\! \bm\xi_j^\top \!Q \bm\xi_j
    \!\leq\! \mc R_\epsilon \!(Q) .
\end{align*}
The limit $\lim_{J\rightarrow\infty} \mc R_\epsilon(Q) -\! \Delta_Q J^{-\frac{1}{d}} $ is equal to $\mc R_\epsilon(Q)$. Therefore, using the sandwich theorem, the supremum of the piece-wise linear approximation is squeezed into the equality
\begin{align*}
    \lim_{J\rightarrow\infty} \sup_{\mathbb{Q} \in \mathbb{B}_{\epsilon}} \! \bb E_{\xi_T \sim \bb Q} \max_{j \in [J]} \, 2\bm\xi_T^\top Q \bm\xi_j \!-\! \bm\xi_j^\top Q \bm\xi_j
    = \mc R_\epsilon(Q),
\end{align*}
and the limit exists.

The second part of the proof aims at bringing the limit back into the problem and evaluating it. Using the previous result and Proposition \ref{prop:cvar_original_deterministic_constraint_bounded_support}  for all $J \in \bb N$, with $a_j = 2Q \bm\xi_j$ and $b_j = -\bm\xi_j^\top Q \bm\xi_j$, we know that $\mc R_\epsilon(Q)$ as defined in \eqref{eq:prop_dro_general_quad_risk_def} is equal to
\begin{align}
    \lim_{J\rightarrow\infty} &\inf_{s^{(i)},
    \lambda \geq 0, \kappa_{ij} \geq 0}  
    \lambda \epsilon + \frac{1}{N} \!\sum_{i \in [N]} s^{(i)}\,,\;\st
    \nonumber \\
    &\quad s^{(i)} \geq f(\bm\xi_j, \kappa_{ij}, \lambda) \,, \forall i\in[N] \,, \forall j \in [J] \label{eq:proof_prop10_second_part}\,,
\end{align}
where 
\begin{align*}
    &f(\bm\xi, \kappa, \lambda) = -\bm\xi^\top Q \bm\xi \!+\! \frac{1}{\lambda}\bm\xi^\top Q^2 \bm\xi  + 2 \bm\xi^\top Q \bm \xi_T^{(i)}
    \\
    &\quad\quad\; +\! \frac{1}{4\lambda}\!\lt( \kappa^\top H H^\top \kappa \!-\! 4\bm\xi^\top Q H^\top \kappa \rt)\! + \big(h - H \bm \xi_T^{(i)} \big)^{\!\!\top}\! \kappa \,,
\end{align*}
which is convex in $\kappa, \lambda \geq 0$ for any given $\bm\xi$ because it is the sum of constant, linear, inverse, and inverse times quadratic functions.
Since $\kappa_{ij} > 0$ affects only the constraint \eqref{eq:proof_prop10_second_part}, one can write the existence constraint as a constraint on the minimum, which yields
\begin{subequations}\label{eq:prop_dro_general_quad_risk_dual_inequality_discretized}
\begin{align}\label{eq:prop_dro_general_quad_risk_dual_inequality_discretized_cost}
    \!\!\lim_{J\rightarrow\infty} &\inf_{
    s^{(i)}, \lambda \geq 0}  
    \lambda \epsilon + \frac{1}{N} \!\sum_{i \in [N]} s^{(i)}\,,\; \st
    \\ \label{eq:prop_dro_general_quad_risk_dual_inequality_discretized_cons}
    &\quad s^{(i)} \geq \min_{\kappa_{ij} \geq 0} f(\bm\xi_j, \kappa_{ij}, \lambda) \,, \forall i\in[N] \,, \forall j \in [J]\,.\!\!
\end{align}
\end{subequations}

In the rest of the proof, we relate \eqref{eq:prop_dro_general_quad_risk_dual_inequality_discretized} to the following semi-infinite problem.
\begin{subequations}\label{eq:prop_dro_general_quad_risk_dual_inequality_seminf}
\begin{align}\label{eq:prop_dro_general_quad_risk_dual_inequality_seminf_cost}
    &\inf_{
    s^{(i)}, \lambda \geq 0}  
    \lambda \epsilon + \frac{1}{N} \!\sum_{i \in [N]} s^{(i)}\,,\; \st
    \\ \label{eq:prop_dro_general_quad_risk_dual_inequality_seminf_cons}
    &\quad s^{(i)} \geq \min_{\kappa_{i} \geq 0} f(\bar{\bm\xi}, \kappa_{i}, \lambda) \,, \forall i\in[N] \,, \forall \bar{\bm\xi} \in \Xi \,.
\end{align}
\end{subequations}
Because \eqref{eq:prop_dro_general_quad_risk_dual_inequality_seminf_cons} is sufficient for \eqref{eq:prop_dro_general_quad_risk_dual_inequality_discretized_cons}, we observe that $\eqref{eq:prop_dro_general_quad_risk_dual_inequality_seminf} \geq \eqref{eq:prop_dro_general_quad_risk_dual_inequality_discretized} = \mc R_\epsilon(Q)$ in general. Moreover, from \cite[Corollary 3.1]{shapiro2009semiinfinite}, if \eqref{eq:prop_dro_general_quad_risk_dual_inequality_seminf} is convex and its solution set is nonempty and bounded, then it is discretizable, meaning that it can be approximated up to arbitary precision by \eqref{eq:prop_dro_general_quad_risk_dual_inequality_discretized} with a fine enough grid $\mc G_J$. Note that \eqref{eq:prop_dro_general_quad_risk_dual_inequality_seminf} is linear except for \eqref{eq:prop_dro_general_quad_risk_dual_inequality_seminf_cons} so we only need \eqref{eq:prop_dro_general_quad_risk_dual_inequality_seminf_cons} to be convex for all $i=1,\dots,N$, which is guaranteed by the convexity of $f$ for any given $\bm\xi$. Moreover, (i) $\bm\Xi$ is bounded, so there always exist large enough yet finite $\lambda$ and $s^{(i)}$ such that \eqref{eq:prop_dro_general_quad_risk_dual_inequality_seminf} is feasible with $\kappa_i = 0$, (ii) $\lambda$ and $s^{(i)}$ are penalized in the cost so any unbounded value has a higher cost than the aforementioned feasible point, and (iii) $f$ is radially unbounded in $\kappa$ so if $s^{(i)}$ is bounded, so is $\kappa_i = 0$. Hence, so the solution set is nonempty and bounded, which means that \eqref{eq:prop_dro_general_quad_risk_dual_inequality_seminf} is discretizable and therefore equal to \eqref{eq:prop_dro_general_quad_risk_dual_inequality_discretized}.

Finally, semi-infinite problems can be rewritten in a bi-level form, which implies that \eqref{eq:prop_dro_general_quad_risk_dual_inequality_seminf} is equal to
\begin{subequations}
\label{eq:prop_dro_general_quad_risk_dual_maxmin}
\begin{align}
    &\inf_{
    s^{(i)}, \lambda \geq 0}  
    \lambda \epsilon + \frac{1}{N} \!\sum_{i \in [N]} s^{(i)}\,,\; \st
    \\ \label{eq:prop_dro_general_quad_risk_dual_inequality_maxmin}
    &\quad s^{(i)} \geq \max_{\bar{\bm\xi} \in \Xi} \min_{\kappa_{i} \geq 0} f(\bar{\bm\xi}, \kappa_{i}, \lambda) \,, \forall i\in[N] \,.
\end{align}
\end{subequations}
In general, one has $\max_{\bar{\bm\xi} \in \Xi} \min_{\kappa_{i} \geq 0} f(\bar{\bm\xi}, \kappa_{i}, \lambda) \leq \min_{\kappa_{i} \geq 0} \max_{\bar{\bm\xi} \in \Xi} f(\bar{\bm\xi}, \kappa_{i}, \lambda)$. This means that \eqref{eq:prop_dro_general_quad_risk_dual_inequality} is a stricter constraint than \eqref{eq:prop_dro_general_quad_risk_dual_inequality_maxmin}, yielding $\eqref{eq:prop_dro_general_quad_risk_dual_maxmin} \leq \eqref{eq:prop_dro_general_quad_risk_dual}$. Nevertheless, if $f$ is not only convex in $\kappa$ but also concave in $\bm \xi$, i.e., $Q - \lambda^{-1}Q^2 \succeq 0$, then Sion's minimax theorem proves that the $\max$ and $\min$ operators commute \cite[Corollary 3.3]{sion1958general}. This means that if $\lambda I \succeq Q$, \eqref{eq:prop_dro_general_quad_risk_dual_inequality} and \eqref{eq:prop_dro_general_quad_risk_dual_inequality_maxmin} are equivalent, which concludes the proof.
\end{proof}
\vspace{10pt}

Using Proposition \ref{prop:quad_dro_prob_is_limit_linear}, we are now ready to prove Lemma \ref{coro:quad_dro_prob_is_limit_linear} by dualizing \eqref{eq:prop_dro_general_quad_risk_dual_inequality} to remove the $\max$ operator, and by using Schur's complement to obtain linear inequalities. We start by highlighting that \eqref{eq:prop_dro_general_quad_risk_dual_inequality} contains the maximization of the quadratic cost 
\begin{align*}
    -\underbrace{\bar{\bm\xi}^\top (Q - \lambda^{\!-1}Q^2) \bar{\bm\xi}}_{\mt{quadratic}} 
    & + \underbrace{\bar{\bm\xi}^\top ( 2 Q \bm \xi_T^{(i)} \!-\! \lambda^{\!-1} Q H^\top \kappa_i)}_{\mt{linear}}
    \\ \nonumber
    &+\! \underbrace{\frac{1}{4\lambda} \kappa_i^\top H H^\top \kappa_i + \big(h - H \bm \xi_T^{(i)}  \big)^{\!\!\top} \kappa_i}_{\mt{constant}} \,,
\end{align*}
subject to convex polytopic constraints $H\bar{\bm\xi} - h \leq 0$. The dual problem of this QP is therefore given by \cite{boyd2004convex} as
\begin{subequations}
\begin{align}
    \!\!\min_{\substack{\mu_i \geq 0}} 
    \!
    \mu_i^\top h & \!+\! \overbrace{\frac{1}{4\lambda} \kappa_i^{\!\top} \!H\! H^{\!\top}\!\! \kappa_i \!+\! \big(h \!-\! H \bm \xi_T^{(i)} \big)^{\!\!\top}\!\! \kappa_i}^{(\clubsuit)}
    \nonumber \\
    &\!+\! \frac{1}{4} \Big\| H^{\!\top} \!\!\mu_i \!+\! \frac{1}{\lambda}\!(HQ)^{\!\top}\!\!\kappa_i \!-\! 2 Q \bm\xi_T^{(i)} \!\Big\|^{2}_{Q_2} \!
    ,\!\! \label{eq:proof_coro_quad_dro_prob_is_limit_linear_dual}
    \\ \label{eq:proof_coro_quad_dro_prob_is_limit_linear_dual_constraint}
    \st\; & P_{\lambda} \Big(H^{\!\top} \!\!\mu_i \!+\! \frac{1}{\lambda}\!(HQ)^{\!\top}\!\!\kappa_i \!-\! 2 Q \bm\xi_T^{(i)}\Big) = 0,
\end{align}
\end{subequations}
where $\|\cdot\|_{Q_2}^2 = \cdot^\top Q_2 \,\cdot$, $Q_2 = (Q - \lambda^{\!-1}Q^2)^{\dagger}$, and $P_\lambda$ is the projection on $\operatorname{null}(Q_2^\dagger) = \operatorname{null}(\lambda I - Q)$. Note that $P_\lambda =  I - (\lambda I - Q)^\dagger (\lambda I - Q)$ is symmetric, commutes with $Q$ and $Q^{-1}$, and is equal to both its square and pseudo-inverse. Since we are looking for an upper bound for $\mc R_\epsilon(Q)$ when $\lambda \leq \lambda_{\max}(Q)$, we can replace \eqref{eq:proof_coro_quad_dro_prob_is_limit_linear_dual_constraint} by the stricter constraint
\begin{align}\label{eq:proof_coro_quad_dro_prob_is_limit_linear_dual_constraint_strict} 
    P_{\lambda} H^{\!\top} \!\!\mu_i = 0\,, P_{\lambda} \big( H^{\!\top}\!\!\kappa_i \!-\! 2 \lambda \bm\xi_T^{(i)} \big) = 0,
\end{align}
as it leads to a larger minimum if $P_\lambda \neq 0$ and but is equivalent for any $\lambda > \lambda_{\max}(Q)$ because $P_\lambda = 0$.
Moreover, the last term of \eqref{eq:proof_coro_quad_dro_prob_is_limit_linear_dual} can be split as
\begin{align*}
    & \frac{1}{4} \Big\| 2 Q \bm\xi_T^{(i)} \!-\!  \frac{1}{\lambda}\!(HQ)^{\!\top} \!\!\kappa_i  \Big\|_{Q_2}^2 
    \\
    &\quad -\frac{1}{2} \!\Big(\! 2 Q \bm\xi_T^{(i)} \!-\!  \frac{1}{\lambda}\!(HQ)^{\!\top} \!\!\kappa_i  \!\Big)^{\!\!\top}\!\! Q_2 H^{\!\top} \!\!\mu_i
    + \frac{1}{4}\mu_i^{\!\top}\! H Q_2 H^{\!\top} \!\!\mu_i\,,
\end{align*}
or equivalently,
\begin{subequations}
\label{eq:coro_proof_quad_split}
\begin{align}
\label{eq:coro_proof_quad_split_a}
    & \frac{1}{4} ( 2 \lambda \bm\xi_T^{(i)} \!-\! H^{\!\top} \!\!\kappa_i )^\top (\lambda^2 Q^{-1} \!-\! \lambda I)^\dagger ( 2 \lambda \bm\xi_T^{(i)} \!-\! H^{\!\top} \!\!\kappa_i )
    \\
\label{eq:coro_proof_quad_split_b}
    &-\frac{1}{2} \!\Big(\! 2 \lambda \bm\xi_T^{(i)} \!+\! H^{\!\top} \!\!\kappa_i  \!\Big)^{\!\!\top}\!\! (\lambda I - Q)^\dagger H^{\!\top} \!\!\mu_i
    \\
\label{eq:coro_proof_quad_split_c}
    &+ \frac{1}{4}\mu_i^{\!\top}\! H Q_2 H^{\!\top} \!\!\mu_i\,,
\end{align}
\end{subequations} 

In order to obtain some simplifications, we use the following Woodbury-like identities:
\begin{subequations}
\label{eq:coro_proof_woodbury}
\begin{align}
(\lambda^2Q^{-1} \!-\! \lambda I)^\dagger &= \!\frac{1}{\lambda^2} \!\! \lt(\! Q^{-1} \!-\! \frac{1}{\lambda} I \!\rt)^\dagger,
\nonumber \\
& = \!\!\lt(\! \frac{1}{\lambda} Q^{-1} \!\!-\! \frac{1}{\lambda} Q^{-1} \!\!+\! \frac{1}{\lambda^2}I \!\rt)\!\!\lt(\! Q^{-1} \!-\! \frac{1}{\lambda} I \!\rt)^{\!\!\dagger},
\nonumber \\
& = \!(\lambda I \!-\! Q)^\dagger \!+\! \!\lt(\! \frac{1}{\lambda^2}I \!-\! \frac{1}{\lambda} Q^{-1} \!\!\rt)\!\!\lt(\! Q^{-1} \!-\! \frac{1}{\lambda} I \!\rt)^{\!\!\dagger},
\nonumber \\
& = \!(\lambda I \!-\! Q)^\dagger \!-\! \frac{1}{\lambda}(I - P_\lambda),
\label{eq:coro_proof_woodbury_a}
\end{align}
and
\begin{align}\label{eq:coro_proof_woodbury_b}
Q_2 &= \lambda Q^{-1} (\lambda I - Q)^{\dagger},
\\
&= (I + \lambda Q^{-1} - I) (\lambda I - Q)^{\dagger},
\nonumber \\
&= (\lambda I - Q)^{\dagger} + Q^{-1}(I - P_\lambda),
\nonumber \\ \label{eq:coro_proof_woodbury_c}
&= (\lambda I - Q)^{\dagger} + Q^{-1} - P_\lambda Q^{-1} P_\lambda.
\end{align}
\end{subequations} 
We plug \eqref{eq:coro_proof_woodbury_a}, \eqref{eq:coro_proof_woodbury_b}, and \eqref{eq:coro_proof_woodbury_c} into \eqref{eq:coro_proof_quad_split_a}, \eqref{eq:coro_proof_quad_split_b}, and \eqref{eq:coro_proof_quad_split_c}, respectively, which gives
\vspace{-12pt}
\begin{subequations}
\label{eq:coro_proof_quad_split_with_woodbury}
\begin{align}
\nonumber
    \!\!&\eqref{eq:coro_proof_quad_split} \!=\! \frac{1}{4} \Big\| 2 \lambda \bm\xi_T^{(i)} \!\!-\!  H^{\!\top} \!\!\kappa_i  \Big\|_{(\lambda I \!- Q)^{\!\dagger}}^2 \!\!\!-\! \overbrace{\frac{1}{4\lambda} \Big\| 2 \lambda \bm\xi_T^{(i)} \!\!-\!  H^{\!\top} \!\!\kappa_i  \Big\|_2^2}^{(\spadesuit)}\!\!
    \\
    \label{eq:coro_proof_quad_split_with_woodbury_a}
    &\quad\quad\quad +\! \lt. \frac{1}{4\lambda} \Big\| 2 \lambda \bm\xi_T^{(i)} \!\!-\!  H^{\!\top} \!\!\kappa_i  \Big\|_{P_\lambda}^2 \rt\}(\bigstar)
    \\
\label{eq:coro_proof_quad_split_with_woodbury_b}
    &-\! \frac{1}{2} \!\Big(\! 2 \lambda \bm\xi_T^{(i)} \!-\!  H^{\!\top} \!\!\kappa_i \!\Big)^{\!\!\top}\!\! (\lambda I \!-\! Q)^{\!\dagger} H^{\!\top} \!\!\mu_i
    \\
\label{eq:coro_proof_quad_split_with_woodbury_c}
    &+\! \frac{1}{4}\| H^{\!\top} \!\!\mu_i\|_{(\lambda I \!- Q)^{\!\dagger}}^2 \!\!+\! \underbrace{\frac{1}{4}\| H^{\!\top} \!\!\mu_i\|_{Q^{-1}}^2 \!}_{(\blacklozenge)} \!-\! \underbrace{ \frac{1}{4}\| P_\lambda H^{\!\top} \!\!\mu_i\|_{Q^{-1}}^2 }_{(\bigstar)}\!.
\end{align}
\end{subequations} 
The terms $(\clubsuit)$ in \eqref{eq:proof_coro_quad_dro_prob_is_limit_linear_dual} can be grouped by completing the squares as
\begin{align}\label{eq:coro_proof_quad_completing}
    \underbrace{\frac{1}{4\lambda} \Big\| 2 \lambda \bm\xi_T^{(i)} \!-\!  H^{\!\top} \!\!\kappa_i  \Big\|_2^2}_{(\spadesuit)} \!- \lambda\|\bm\xi_T^{(i)}\|_2^2 + \kappa_i^\top h \,,
\end{align}
We remark that the terms marked by $(\spadesuit)$ in \eqref{eq:coro_proof_quad_completing} and \eqref{eq:coro_proof_quad_split_with_woodbury_c} cancel out, and that the terms marked by $(\bigstar)$ in \eqref{eq:coro_proof_quad_split_with_woodbury} can be factorized as
\begin{align}\label{eq:coro_proof_quad_completing_2}
    \!\!( H^{\!\top} \!\!\kappa_i \!-\! 2 \lambda \bm\xi_T^{(i)}\! \!+\! H^{\!\top} \!\!\mu_i)^{\!\top}\! P_\lambda \Big( \frac{1}{\lambda} H^{\!\top} \!\!\kappa_i \!-\! 2\bm\xi_T^{(i)} \!\!-\! Q^{-1} H^{\!\top} \!\!\mu_i \Big),\!\!
\end{align}
because the added cross terms are in the null space of $(\lambda I - Q)$. The constraint \eqref{eq:proof_coro_quad_dro_prob_is_limit_linear_dual_constraint} implies that \eqref{eq:coro_proof_quad_completing_2} is zero, so the terms marked by $(\bigstar)$ in \eqref{eq:coro_proof_quad_split_with_woodbury} cancel out. Finally, all remaining terms besides $(\blacklozenge)$ in \eqref{eq:coro_proof_quad_split_with_woodbury_c} can be factorized. Hence, the dual problem \eqref{eq:proof_coro_quad_dro_prob_is_limit_linear_dual} is equal to
\begin{align} \label{eq:proof_coro_quad_dro_prob_is_limit_linear_dual_simplified}
    \min_{\substack{\mu_i \geq 0}} 
    h^{\!\top}\! (\kappa_i \!+\! \mu_i) &+ \frac{1}{4}\| H^{\!\top} \!\!\mu_i\|_{Q^{-1}}^2 \!-\! \lambda\|\bm\xi_T^{(i)}\|_2^2
    \\ \nonumber
    &+\! \frac{1}{4} \Big\| 2 \lambda \bm\xi_T^{(i)} \!\!-\!  H^{\!\top} \!\!(\kappa_i \!+\! \mu_i)  \!\Big\|^{2}_{(\lambda I \!- Q)^{\!\dagger}} 
    ,
\end{align}

In general, the right-hand side of \eqref{eq:prop_dro_general_quad_risk_dual_inequality} is smaller than \eqref{eq:proof_coro_quad_dro_prob_is_limit_linear_dual_simplified}, which means that $s^{(i)} \geq$ \eqref{eq:proof_coro_quad_dro_prob_is_limit_linear_dual_simplified} implies \eqref{eq:prop_dro_general_quad_risk_dual_inequality}. Moreover, if $\lambda I - Q \succeq 0$, the problem \eqref{eq:proof_coro_quad_dro_prob_is_limit_linear_dual_simplified} is a convex and strictly feasible QP. Strong duality therefore shows that the right-hand side of \eqref{eq:prop_dro_general_quad_risk_dual_inequality} is equal to \eqref{eq:proof_coro_quad_dro_prob_is_limit_linear_dual_simplified} in this case. Finally, we replace the upper bound on a minimum by an existence constraint and perform the change of variable $\psi_i = \kappa_i + \mu_i$ to rewrite \eqref{eq:prop_dro_general_quad_risk_dual_inequality} as
\begin{align*}
    & s^{(i)} \geq  h^{\!\top}\! \psi_i \!-\! \lambda\|\bm\xi_T^{(i)}\|_2^2 \!+\! \aoverbrace[L1R]{\frac{1}{4} \mu_i^{\!\top} H Q^{-1} H^{\!\top} \!\!\mu_i}
    \\ \nonumber
    &\quad\;\;
    \!+\! \aunderbrace[l1r]{\frac{1}{4} \big( 2 \lambda \bm\xi_T^{(i)} \!\!-\!  H^{\!\top} \!\psi_i  \big)^{\!\!\top}\!\! (\lambda I \!-\! Q)^{\!\dagger} \!\big( 2 \lambda \bm\xi_T^{(i)} \!\!-\!  H^{\!\top} \!\psi_i  \big)}\!.
\end{align*}
Applying Schur's lemma to the two terms highlighted with brackets and with \eqref{eq:proof_coro_quad_dro_prob_is_limit_linear_dual_constraint_strict}, we obtain
\begin{align*}
    \mc R_\epsilon(Q) \leq \inf_{\substack{
    \lambda \geq 0, \mu_i \geq 0, \\ s^{(i)}\!,\, \psi_i \geq \mu_i}}  
    &
    \lambda \epsilon + \frac{1}{N} \!\sum_{i \in [N]} s^{(i)}\,,\; \\ \nonumber
    &\st\,, \forall i \in [N]:
    \\
    & P_\lambda H^\top \mu_i = 0, 
    \\
    &\hspace{-32pt}\lt[\matb
    \!s^{(i)} \!\!-\!  h^{\!\top}\! \psi_i
    \!+\! \lambda\|\bm\xi_T^{(i)} \! \|_2^2 \hspace{-8pt}\!
    &
    \star
    &
    \!\!\!\!\star
    \\
    2 \lambda \bm\xi_T^{(i)} \!\!-\!  H^{\!\top} \!\psi_i
    &
    4(\lambda I \!-\! Q)
    &
    \!\!\!\!\star 
    \\
    H^{\!\top} \mu_i
    &
    0
    &
    \!\!\!\!4Q\!
    \mate\rt] \! \!\succeq \! 0, \!\!
\end{align*}
where the equality holds when $\lambda I - Q \succeq 0$. Finally, the constraint $P_\lambda H^\top \mu_i = 0$ can be enforced as LMI using Schur's complement of $\alpha - \mu_i^\top\! H (\lambda I - Q)^\dagger H^\top\! \mu_i$ with an arbitrarily large $\alpha$, which concludes the proof.

\subsection{Proof of Lemma \ref{lem:drinc_cost}}
\label{app:proof_lem_drinc_cost}

The proof is conducted in three parts. First, we rewrite the quadratic form $\bm\Phi^{\!\top} \! D \bm\Phi$ as a matrix $Q$ to obtain linear constraints. Second, we analyze the suboptimality when $Q \succ 0$ and show that it vanishes when $Q \rightarrow \bm\Phi^{\!\top} \! D \bm\Phi$. Third and finally, we rewrite all the constraints as LMIs.

We start by showing that
\begin{align}\label{eq:proof_lem_drinc_cost_min_Q}
\mc R_\epsilon(\bm\Phi^{\!\top} \! D \bm\Phi) = \min_{Q \succeq \bm\Phi^{\!\top} \!D \bm\Phi} \mc R_\epsilon(Q).
\end{align}
Recall the definition 
\begin{align*}
\mc R_\epsilon(Q) := \sup_{\mathbb{Q} \in \mathbb{B}_{\epsilon}} \bb E_{\xi_T \sim \bb Q} ~  \bm\xi_T^\top Q \bm\xi_T,
\end{align*}
and note that for any $\bm \xi_T \in \bm \Xi$, if $Q \succeq \bm \Phi^{\!\top} \! D \bm \Phi$ the following inequality holds
\begin{align*}
\bm\xi_T^\top Q \bm\xi_T \geq \bm\xi_T^\top \bm \Phi^{\!\top} \!D \bm \Phi \bm\xi_T.
\end{align*}
Hence, because probability distributions are non-negative and integrals preserve the order, one has
\begin{align*}
\bb E_{\bm \xi_T\sim \bb Q}\lt[  \bm\xi_T^\top Q \bm\xi_T\rt] 
\geq 
\bb E_{\bm \xi_T\sim \bb Q}[\bm\xi_T^\top \bm \Phi^{\!\top} \!D \bm \Phi \bm\xi_T],
\end{align*}
for any probability distribution $\bb Q$ and therefore also for the worst one. Hence, $Q \succeq \bm \Phi^{\!\top} \!D \bm \Phi$ implies that $\mc R_\epsilon(Q) \geq \mc R_\epsilon(\bm \Phi^{\!\top} \! D \bm \Phi)$. Moreover, the equality is attained because $\bm \Phi^{\!\top} \! D \bm \Phi \in \argmin_{Q \succeq \bm\Phi^{\!\top} \! D \bm\Phi} \mc R_\epsilon(Q)$. 

The proof continues by showing
\begin{align}\label{eq:proof_lem_drinc_cost_lim_eta}
    \mc R_\epsilon(\bm\Phi^{\!\top} \! D \bm\Phi) = \min_{Q \succeq \bm \Phi^{\!\top} \! D \bm \Phi} \lim_{\eta \rightarrow 0} \mc R_\epsilon(Q+|\eta| I).
\end{align}
Note that $\mc R_\epsilon(Q) = \mc R_\epsilon(\lim_{\eta \rightarrow 0} Q+|\eta| I)$, where one can take the limit out of the risk using the inequality
\begin{align}\label{eq:prop_proof_lem_drinc_cost_squeeze}
    \mc R_\epsilon(Q) + |\eta| \max_{\bm \xi_T \in \bm \Xi} \|\bm \xi_T\|_2^2 \geq \mc R_\epsilon(Q + |\eta| I) \geq \mc R_\epsilon(Q),
\end{align}
which holds if $\bm \Xi$ is bounded. This means that the limit for $\eta \rightarrow 0$ is squeezed between two values that tend towards $\mc R_\epsilon(Q)$.

We finish the proof by expressing $Q \succeq \bm\Phi^{\!\top} \! D \bm\Phi$ as a Schur complement. This yields
\begin{align}
\label{eq:proof_quad_cost_coro_end}
    \lt[\matb
    Q - \eta I & \Phi^{\!\top} D^{\frac{1}{2}}
    \\
    D^{\frac{1}{2}} \Phi & \alpha I
    \mate\rt]
    \succeq 0 \,.
\end{align}
Combining \eqref{eq:proof_lem_drinc_cost_lim_eta} and \eqref{eq:proof_quad_cost_coro_end} yields \eqref{eq:lem_quad_cost}, which concludes the proof.

\subsection{Proof of Lemma \ref{lem:drinc_constraints}}
\label{app:proof_lem_drinc_constraints}

Lemma \ref{lem:drinc_constraints} is a direct consequence of applying Proposition \ref{prop:cvar_original_deterministic_constraint_bounded_support} to the definition \eqref{eq:cvar_constraint_definition}. Indeed, with $G_{J+1} = 0$ and $ g_{J+1} = -\tau$, one can rewrite \eqref{eq:conditional_value_at_risk_safety_constraints_inf_horizon_both} as \eqref{eq:prop_dro_general_risk_def} by setting 
$a_j = \gamma^{-1}G_{j}\bm\Phi, b_j = \gamma^{-1}(-g_{j} - \tau + \gamma \tau)$. This means that \eqref{eq:conditional_value_at_risk_safety_constraints_inf_horizon_both} is equivalent to
\begin{align*}
    &\inf_{s^{(i)},
    \rho \geq 0, \kappa_{ij} \geq 0}  
    \rho \epsilon + \frac{1}{N} \!\sum_{i \in [N]} s^{(i)} \leq 0 \,,
    \\
    &\st \,, \forall i \in [N] \,, \forall j \in [J+1] :
    \\
    & s^{(i)} \geq \frac{1}{\gamma}(-g_{j} \!-\! \tau \!+\! \gamma \tau) \!+\! \frac{1}{\gamma}G_{j}^\top \bm\Phi \bm \xi_T^{(i)}  \!+\! \big(h \!-\! H \bm \xi_T^{(i)} \big)^{\!\!\top} \! \kappa_{ij}
    \\ \nonumber
    &\quad\quad\quad +\! \frac{\|H^{\!\top}\! \kappa_{ij}\|_2^2}{4\rho}  \!-\! \frac{1}{2\rho\gamma} G_{j}^\top \bm\Phi H^{\!\top}\! \kappa_{ij} \!+\! \frac{\|\bm\Phi^{\!\top} G_{j}\|_2^2}{4\rho\gamma^2}.
\end{align*}
One can factorize the last three terms of the constraint and do the change of variable $\zeta^{(i)} = s^{(i)} + \gamma^{-1}\tau - \tau$, which gives
\begin{subequations}
\begin{align}
    &\inf_{s^{(i)},
    \rho \geq 0, \kappa_{ij} \geq 0}  
    \rho \epsilon - \frac{1}{\gamma}\tau + \tau + \frac{1}{N} \!\sum_{i \in [N]} \zeta^{(i)} \leq 0 \,,
    \\ \nonumber
    &\st \,, \forall i \in [N] \,, \forall j \in [J+1] :
    \\ \label{eq:proof_lem_constraint_factorized_si}
    &\zeta^{(i)} \geq -\frac{1}{\gamma}g_{j} \!+\! \frac{1}{\gamma}G_{j}^\top \bm\Phi \bm \xi_T^{(i)}  \!+\! \big( h - H \bm \xi_T^{(i)} \big)^{\!\!\top} \! \kappa_{ij}
    \\ \nonumber 
    &\quad\quad\quad +\! \frac{1}{4\rho\gamma^2} (\bm\Phi^{\!\top} G_{j} \!-\! \gamma H^{\!\top}\! \kappa_{ij})^{\!\top}\! (\bm\Phi^{\!\top} G_{j} \!-\! \gamma H^{\!\top}\! \kappa_{ij}).\!\!
\end{align}
\end{subequations}
Finally, a zero upper-bound constraint on an infimum is equivalent to an existence constaint. Moreover, because $\rho \geq 0$, \eqref{eq:proof_lem_constraint_factorized_si} can be written as an LMI using Schur's complement, which concludes the proof.

\end{document}